%% file: giant_almost_local_JAP.tex
\documentclass{imsart}
\RequirePackage{amsthm,amsmath,amsfonts,amssymb}
\RequirePackage[numbers]{natbib}
\RequirePackage{color}
\RequirePackage{mathtools,enumitem}

\RequirePackage{csquotes}

\RequirePackage{subcaption}

\RequirePackage[colorlinks,citecolor=blue,urlcolor=blue]{hyperref}
\RequirePackage{graphicx}% uncomment this for including figures

\startlocaldefs
\numberwithin{equation}{section}
\makeatletter
\newcommand{\customitem}[1]{%
\item[\rm#1]\protected@edef\@currentlabel{#1}%
}
\makeatother

\endlocaldefs

\input{preamble}

\DeclareSymbolFont{extraup}{U}{zavm}{m}{n}
\DeclareMathSymbol{\varheart}{\mathalpha}{extraup}{86}
\DeclareMathSymbol{\vardiamond}{\mathalpha}{extraup}{87}
\newcommand{\ensymboldefinition}{\blacktriangleleft}

\begin{document}

\begin{frontmatter}
\title{The giant in random graphs is almost local}
%\runtitle{The giant is almost local}

\begin{aug}
%%%%%%%%%%%%%%%%%%%%%%%%%%%%%%%%%%%%%%%%%%%%%%
%%Only one address is permitted per author. %%
%%Only division, organization and e-mail is %%
%%included in the address.                  %%
%%Additional information can be included in %%
%%the Acknowledgments section if necessary. %%
%%%%%%%%%%%%%%%%%%%%%%%%%%%%%%%%%%%%%%%%%%%%%%
\author{\fnms{Remco} \snm{van der Hofstad}\ead[label=e1]{r.w.v.d.hofstad@tue.nl}}
%%%%%%%%%%%%%%%%%%%%%%%%%%%%%%%%%%%%%%%%%%%%%%
%% Addresses                                %%
%%%%%%%%%%%%%%%%%%%%%%%%%%%%%%%%%%%%%%%%%%%%%%
\address{Department of Mathematics and Computer Science, Eindhoven University of Technology, \printead{e1}}
\end{aug}

\begin{abstract}
Local convergence techniques have become a key methodology to study sparse random graphs. However, convergence of many random graph properties does not directly follow from local convergence. A notable, and important, such random graph property is the size and uniqueness of the giant component. 

We provide a simple criterion that guarantees that local convergence of a random graph implies the convergence of the proportion of vertices in the maximal connected component. We further show that, when this condition holds, the local properties of the giant, as well as its complement, are also described by the local limit. We give several examples where this method gives rise to a novel law of large numbers for the giant, based on results proved in the literature. Aside from these examples, we apply our method to the classical problem of giants in the configuration model as a proof of concept, reproving a well-established result. As a side result of this proof,  we give an extremely simple proof of the small-world nature of the configuration model.
\end{abstract}



\end{frontmatter}
%%%%%%%%%%%%%%%%%%%%%%%%%%%%%%%%%%%%%%%%%%%%%%
%% Please use \tableofcontents for articles %%
%% with 50 pages and more                   %%
%%%%%%%%%%%%%%%%%%%%%%%%%%%%%%%%%%%%%%%%%%%%%%
% \tableofcontents

%%%%%%%%%%%%%%%%%%%%%%%%%%%%%%%%%%%%%%%%%%%%%%
%%%% Main text entry area:

\section{Introduction}
\label{sec-intro}
{\it Local convergence} techniques, as introduced by Aldous and Steele  \cite{AldSte04} and Benjamini and Schramm \cite{BenSch01}, have become the main methodology to study random graphs in sparse settings, i.e., settings where the average degree remains bounded. Local convergence roughly means that the proportion of vertices whose neighbourhoods have a certain shape converges to some limiting value, which is to be considered as a measure on rooted graphs. We refer to \cite[Chapter 2]{Hofs24} for details. 

The {\em giant component} problem has received enormous attention ever since the very first and seminal results by Erd\H{o}s and R\'enyi on the \erdos{} \cite{ErdRen60}, see also \cite{Boll01, JanKnuLucPit93} for detailed further results for this model. The simplest form of this question is whether there exists a linear-size connected component or not, in many cases a sharp transition occurs depending on a certain graph parameter. See \cite{AloSpe00, ChuLu06c, Durr06, FriKar16, FriKar23, Hofs17, Hofs24, JanLucRuc00}, as well as the references therein, for details.

This paper combines these two threads by investigating the size of the giant component when the random graph converges locally. Consider a sequence of random graphs $(G_n)_{n\geq 1}$, where we will simplify the notation by assuming that $G_n=(V(G_n),E(G_n))$ is such that $|V(G_n)|=n$. For $v\in V(G_n)$, we let $\cluster(v)$ denote its connected component. We let 
	\eqn{
	\label{maximal-component}
	|\Cmax|=\max_{v\in V(G_n)}|\cluster(v)|
	}
denote the maximal component size. When the random graph converges locally to some limit, one would expect that also $|\Cmax|/n\convp \zeta$, where $\zeta$ is the survival probability of the local limit. However, while the {\em number} of connected components is well behaved in the local topology, the {\em proportion of vertices in the giant} is not. Indeed, since local convergence is all about proportions of vertices whose {\em finite} (but arbitrarily large) neighbourhoods converge to a limit, there is a potentially enormous gap between surviving locally and being in the giant. In this paper, we identify the extra condition that is necessary and sufficient for the above natural implication to hold.

This paper has an unusual history, for which we refer to Section \ref{sec-disc-open-problems} for an extensive discussion.

\section{Asymptotics and properties of the giant}
\label{sec-asymp-prop-giant}
In this section, we investigate the behaviour of the giant component $|\Cmax|$ for random graphs that converge locally. In Section \ref{sec-local-conv-in-prob}, we introduce the notion of local convergence in probability. In Section \ref{sec-bds-giant}, we study the asymptotics of $|\Cmax|$ for random graphs that converge locally in probability under necessary and sufficient `giant-is-almost-local' condition, and in Section \ref{sec-prop-giant}, we investigate local properties of the giant. In Section \ref{sec-second-largest}, we provide examples for the law of large numbers of the giant, obtained through bounds proved elsewhere. In Section \ref{sec-GIOL-rev}, we give an alternative for the `giant-is-almost-local' condition that is often more convenient to establish. We start by introducing some useful notation used throughout this paper.

\paragraph{Notation.} We abbreviate left- and right-hand side by lhs and rhs, respectively. 
%We abbreviate uniformly at random as uar, and with respect to as wrt.
%Given two $\mathbb{R}$-valued random variables $X$ and $Y$, we say that $X\stackrel{d}{=}Y$ if  $\prob(X\leq x)=\prob(Y\leq x)$ for all $x\in \mathbb{R}$. 
A sequence of random variables $(X_n)_{n\geq 1}$ {\em converges in probability} to a random variable $X$, denoted as $X_{n}\convp X$, if, for all $\varepsilon>0$, $\prob(|X_{n}-X|>\varepsilon)\rightarrow 0$.
%, and it {\em converges in distribution} to $X$, denoted as $X_{n}\stackrel{d}{\rightarrow}X,$ if $\lim_{n\rightarrow\infty}\prob(X_{n}\leq x)=\Pv(X\leq x)$ for all $x\in \mathbb{R}$ for which $F_{\sss{X}}(x)=\Pv(X\leq x)$ is continuous. 
%A sequence of random variables $(X_{n})_{n\geq 1}$ is said to be {\em tight}, which we denote as $X_n=\Op(1)$, if for all $\varepsilon>0$, there exists $r>0$ such that $\sup_{n\geq 1}\Pv(|X_{n}|>r)<\varepsilon$. 
We write $X_n=\op(1)$ when $X_n\convp 0$. 
A sequence of events $(\mathcal{E}_{n})_{n\geq 1}$ is said to hold {\em with high probability} (whp) if $\lim_{n\rightarrow\infty}\prob(\mathcal{E}_{n})=1$.

\subsection{Local convergence in probability}
\label{sec-local-conv-in-prob}
Local weak convergence was introduced independently by Aldous and Steele in \cite{AldSte04} and Benjamini and Schramm in \cite{BenSch01}. The purpose of Aldous and Steele in \cite{AldSte04} was to describe the local structure of the so-called `stochastic mean-field model of distance', meaning the complete graph with i.i.d.\ exponential edge weights. This local description allowed Aldous to prove the celebrated $\zeta(2)$ limit of the random assignment problem \cite{Aldo01}. Benjamini and Schramm in \cite{BenSch01} instead used local weak convergence to show that limits of planar graphs are with probability one recurrent. Since its conception, local convergence has proved a key ingredient in random graph theory. In this section, we provide some basics of local convergence. For more detailed discussions, we refer the reader to \cite{Bord16} or \cite[Chapter 2]{Hofs24}. Let us start with some definitions.
\smallskip

A {\em rooted graph} is a pair $(G,\vertex)$, where $G=(V(G),E(G))$ is a graph with vertex set $V(G)$ and edge set $E(G)$, and $\vertex\in V(G)$ is a vertex. 
Further, a rooted or non-rooted graph is called {\em locally finite} when each of its vertices has finite degree (though not necessarily uniformly bounded). Two (finite or infinite) graphs $G_1=(V(G_1),E(G_1))$ and $G_2=(V(G_2),E(G_2))$ are called {\em isomorphic}, which we write as  $G_1\simeq G_2$, when there exists a bijection $\phi\colon V(G_1)\mapsto V(G_2)$ such that $\{u,v\}\in E(G_1)$ precisely when $\{\phi(u),\phi(v)\}\in E(G_2).$  Similarly, two rooted (finite or infinite) graphs $(G_1, \vertex_1)$ and $(G_2,\vertex_2)$, with $G_i=(V(G_i),E(G_i))$ for $i\in\{1,2\}$, are called {\em isomorphic}, abbreviated as $(G_1, \vertex_1)\simeq (G_2,\vertex_2)$, when there exists a bijection $\phi\colon V(G_1)\mapsto V(G_2)$ such that $\phi(\vertex_1)=\vertex_2$ and $\{u,v\}\in E(G_1)$ precisely when $\{\phi(u),\phi(v)\}\in E(G_2).$ These notions can easily be adapted to multi-graphs (which we will need below), for which $G=(V(G), (x_{i,j})_{i,j\in V(G)})$, where $x_{i,j}$ denotes the number of edges between $i$ and $j$, and $x_{i,i}$ the number of self-loops at $i$. There, instead, the isomorphism $\phi\colon V(G)\mapsto V(G')$ is required to satisfy that $x_{i,j}=x'_{\phi(i),\phi(j)}$, where $(x_{i,j})_{i,j\in V(G)}$ and $(x_{i,j}')_{i,j\in V(G')}$ are the edge multiplicities of $G$ and $G'$, respectively.
\smallskip
 
We let $\mathscr{G}_\star$ be the space of (possibly infinite) connected rooted graphs, where we consider two rooted graphs to be equal when they are isomorphic. Thus, we consider $\mathscr{G}_\star$ as the set of equivalence classes of rooted graphs modulo isomorphisms. The space $\mathscr{G}_\star$ of rooted graphs is a nice metric space, with an explicit metric; see, e.g., \cite[Chapter 2 and Appendix A]{Hofs24} for details.
\smallskip

For a graph $G$ and $u,v\in V(G)$, we let $d_{\sss G}(u,v)$ denote the graph distance between $u$ and $v$, where, by convention, $d_{\sss G}(u,v)=\infty$ when $u$ and $v$ are not connected in $G$. For a rooted graph $(G,\vertex)$, we let $B_r^{\sss(G)}(\vertex)$ denote the (rooted) subgraph of $(G,\vertex)$ of all vertices at graph distance at most $r$ away from $\vertex$. Formally, this means that $B_r^{\sss(G)}(\vertex)=((V(B_r^{\sss(G)}(\vertex)), E(B_r^{\sss(G)}(\vertex)), \vertex)$, where
	\eqan{
	\label{graph-ball}
	V(B_r^{\sss(G)}(\vertex))&=\{u\colon d_{\sss G}(\vertex,u)\leq r\},\\
	E(B_r^{\sss(G)}(\vertex))&=\{\{u,v\}\in E(G)\colon d_{\sss G}(\vertex,u), d_{\sss G}(\vertex,v)\leq r\}.\nn
	}

We say that the graph sequence $(G_n)_{n\geq 1}$ converges {\em locally in probability} to a limit $(G, \vertex)\sim \mu$, when, for every $r\geq 0$ and $H^\star\in \mathscr{G}_\star$,
	\eqn{
	\label{LCP-def}
	\frac{1}{|V(G_n)|} \sum_{v\in V(G_n)} \indic{B_r^{\sss(G_n)}(v)\simeq H^\star} \convp \mu(B_r^{\sss(G)}(o)\simeq H^\star).
	}
This means that the subgraph proportions in the random graph $G_n$ are close, in probability, to those given by $\mu$. Let $\vertex_n\in V(G_n)$ be chosen uniformly at random (uar) in $V(G_n)$. Then, \eqref{LCP-def} is equivalent to the statement that
	\eqn{
	\prob(B_r^{\sss(G_n)}(\vertex_n)\simeq H^\star\mid G_n)\convp \mu(B_r^{\sss(G)}(o)\simeq H^\star).
	}

There are related notions of local convergence, such as {\em local weak convergence}, where \eqref{LCP-def} is replaced by the convergence of {\em expectations}, and  {\em local almost sure convergence}, where the convergence holds almost surely. For our purposes, local convergence in probability is the most convenient, for example since it implies that the neighbourhoods of {\em two} uniformly chosen vertices are asymptotically independent (see e.g., \cite[Corollary 2.18]{Hofs24}), which is central in our proof.

\subsection{Asymptotics of the giant}
\label{sec-bds-giant}
Given a random graph sequence $G_n$ that converges locally in probability to $(G,\vertex)\sim \mu$, one would expect that $|\Cmax|/n\convp \zeta:=\mu(|\cluster(\vertex)|=\infty)$. However, the proportion of vertices in the largest connected component $|\Cmax|/n$ is {\em not} continuous in the local convergence topology, as it is a {\em global} object. It is thus easy to construct counter examples of the fact that $|\Cmax|/n\convp \mu(|\cluster(\vertex)|=\infty)$. However, local convergence still tells us a useful story about the existence of a giant, as well as its size. Indeed, Corollary \ref{cor-giant-UB} shows that the {\em upper bound} is always valid, while Theorem \ref{thm-giant-conv-LWC} shows that a relative simple condition suffices to yield the lower bound as well: 

\begin{corollary}[Upper bound on the giant]
\label{cor-giant-UB}
Let $(G_n)_{n\geq 1}$ be a sequence of graphs having size $|V(G_n)|=n$. %Let $(G,\vertex)$ be a random variable on $\mathscr{G}_\star$ having law $\mu$.
Assume that $G_n$ converges locally in probability to $(G,\vertex)\sim \mu$. Write $\zeta=\mu(|\cluster(\vertex)|=\infty)$ for the survival probability of the limiting graph $(G,\vertex)$.
Then, for every $\vep>0$ fixed, as $n\rightarrow \infty$,
	\eqn{
	\prob(|\Cmax|\leq n(\zeta+\vep))\ra 1.
	}
\end{corollary}

In particular, Corollary \ref{cor-giant-UB} implies that $|\Cmax|/n\convp 0$ when $\zeta=0$, so that there can only be a giant when the local limit has a positive survival probability. Corollary \ref{cor-giant-UB} is folklore in the random graph literature, and has been used in various forms elsewhere (see e.g., \cite{BolJanRio07, BolRio15}). In most of the works where it is used, however, the local limit is a (branching process) tree. This turns out not to be necessary.  

\proof Define
	\eqn{
	\label{Zk-def}
	Z_{\sss \geq k}=\sum_{v\in V(G_n)} \indic{|\cluster(v)|\geq k}.
	}
Assume that $G_n$ converges locally in probability to $(G,\vertex)\sim \mu$ as defined in \eqref{LCP-def}. 
Since $\{|\cluster(v)|\geq k\}=\{|B_k^{\sss(G)}(v)|\geq k\}$, with $\zeta_{\sss\geq k}=\mu(|\cluster(\vertex)|\geq k)$,
	\eqn{
	\label{Zgeqk-conv}
	\frac{Z_{\sss \geq k}}{n}\convp \zeta_{\sss\geq k}.
	}
For every $k\geq 1$,
	\eqn{
	\{|\Cmax| \geq k\}= \{Z_{\sss \geq k}\geq k\},
	}
and, on the event that $Z_{\sss \geq k}\geq 1 $, also $|\Cmax|\leq Z_{\sss \geq k}$. Note that $\zeta=\lim_{k\rightarrow \infty} \zeta_{\sss\geq k}=\mu(|\cluster(\vertex)|=\infty)$.
We take $k$ so large that $\zeta\geq \zeta_{\sss\geq k}-\vep/2$. Then, for every $k\geq 1$, $\vep>0$, and all $n$ large enough such that $k\leq n(\zeta+\vep)$, 
	\eqan{
	\prob(|\Cmax|\geq n(\zeta+\vep))&\leq \prob(Z_{\sss \geq k}\geq n(\zeta+\vep))\\
	&\leq \prob(Z_{\sss \geq k}\geq n(\zeta_{\sss\geq k}+\vep/2))=o(1),\nn
	}
as required.
\qed
\medskip

We conclude that while local convergence cannot determine the size of the largest connected component, it {\em does} prove an upper bound on $|\Cmax|$. There are many results that extend this to $|\Cmax|/n\convp \zeta=\mu(|\cluster(\vertex)|=\infty)$ (see Section \ref{sec-disc-open-problems} for pointers to the literature), but this is no longer a consequence of local convergence alone. Therefore, in general, more involved arguments must be used. We next prove that {\em one}, relatively simple, condition suffices. In its statement, and for $x,y\in V(G_n)$, we write $x\nc y$ when $x\not\in \cluster(y)$: 

\begin{theorem}[The giant is almost local]
\label{thm-giant-conv-LWC}
Let $G_n=(V(G_n),E(G_n))$ denote a random graph of size $|V(G_n)|=n$. %Let $(G,\vertex)$ be a random variable on $\mathscr{G}_\star$ having law $\mu$.
Assume that $G_n$ converges locally in probability to $(G,\vertex)\sim \mu$. %Write $\zeta=\mu(|\cluster(\vertex)|=\infty)$ for the survival probability of the limiting graph $(G,\vertex)$. 
Assume further that
	\eqn{
	\label{almost-local-cond}
	\lim_{k\ra\infty} \limsup_{n\rightarrow \infty}\frac{1}{n^2}\expec\Big[\#\big\{(x,y)\in V(G_n)\times V(G_n)\colon |\cluster(x)|, |\cluster(y)|\geq k, x\nc y\big\}\Big]=0.
	}
Then, with $\cluster_{\sss(2)}$ denoting the second largest connected component (with ties broken arbitrarily),
	\eqn{
	\label{Cmax-C2-asymptotitcs-LWC}
	\frac{|\Cmax|}{n}\convp \zeta=\mu(|\cluster(\vertex)|=\infty),
	\qquad\qquad
	\frac{|\cluster_{\sss(2)}|}{n}\convp 0.
	}
\end{theorem}

Theorem \ref{thm-giant-conv-LWC} shows that a relatively mild condition as in \eqref{almost-local-cond} suffices for the giant to have the expected limit. In fact, we will see that it is  {\em necessary} and sufficient (see Remark \ref{rem-nec-suf} below).  %Theorem \ref{thm-giant-conv-LWC} will be most useful when we can easily show that vertices with large components are likely to be connected.
\smallskip

%We now start with the preparations for the proof of Theorem \ref{thm-giant-conv-LWC}. %The sum of squares of these ordered component sizes will play a crucial role in our analysis. 
In the proof below, for $k\geq 1$, it will be convenient to write $X_{n,k}=\opk(1)$ when
	\eqn{
	\lim_{k\ra\infty} \limsup_{n\rightarrow \infty}\prob(|X_{n,k}|>\vep)=0.
	}

\noindent
{\it Proof of Theorem \ref{thm-giant-conv-LWC}.} By Corollary \ref{cor-giant-UB}, it suffices to prove Theorem \ref{thm-giant-conv-LWC} for $\zeta>0$, which we assume from now on. 

Let the vector $(|\cluster_{\sss(i)}|)_{i\geq 1}$ denote the component sizes ordered in size, from large to small with ties broken arbitrarily, so that $\cluster_{\sss(1)}=\Cmax$. Since $|\Cmax|\geq |\cluster_{\sss(i)}|$ for all $i\geq 1$, we can bound
	\eqan{
	\label{Cmax-ratio}
	\frac{|\Cmax|}{n}&\geq 
	\frac{\frac{1}{n^2} \sum_{i\geq 1} |\cluster_{\sss(i)}|^2\indic{|\cluster_{\sss(i)}|\geq k}}{\frac{1}{n} \sum_{i\geq 1} |\cluster_{\sss(i)}|\indic{|\cluster_{\sss(i)}|\geq k}}.
	}
We investigate both numerator and denominator of this expression. For the denominator, by \eqref{Zgeqk-conv},
	\eqn{
	\label{Zgeqk-once-more}
	\frac{1}{n}\sum_{i\geq 1}  |\cluster_{\sss(i)}|\indic{|\cluster_{\sss(i)}|\geq k}
	=\frac{1}{n}\sum_{v\in V(G_n)} \indic{|\cluster(v)|\geq k}=\frac{1}{n}Z_{\sss \geq k}\convp \zeta_{\sss\geq k},
	}
where we recall that $\zeta_{\sss\geq k}=\mu(|\cluster(\vertex)|\geq k)$. Thus, since $\zeta_{\sss\geq k}\rightarrow \zeta$ as $k\rightarrow \infty$,
	\eqn{
	\frac{1}{n}\sum_{i\geq 1}  |\cluster_{\sss(i)}|\indic{|\cluster_{\sss(i)}|\geq k}=\zeta +\opk(1).
	}
For the numerator, we let 
	\eqn{
	\label{Xnk-def}
	X_{n,k}=\frac{1}{n^2} \sum_{i\geq 1} |\cluster_{\sss(i)}|^2\indic{|\cluster_{\sss(i)}|\geq k}-\Big(\frac{1}{n} \sum_{i\geq 1} |\cluster_{\sss(i)}|\indic{|\cluster_{\sss(i)}|\geq k}\Big)^2,
	}
so that the numerator in \eqref{Cmax-ratio} equals
	\eqn{
	\label{sum-squares-Xnk}
	\big(Z_{\sss \geq k}/n\big)^2+X_{n,k}=\zeta^2+X_{n,k}+\opk(1).
	}
It remains to investigate $X_{n,k}$, which we rewrite as
	\eqan{
	\label{rewrite-almost-local-cond}
	X_{n,k}&=\frac{1}{n^2}\sum_{\stackrel{{\large i,j\geq 1}}{i\neq j}}  |\cluster_{\sss(i)}||\cluster_{\sss(j)}|\indic{|\cluster_{\sss(i)}|, |\cluster_{\sss(j)}|\geq k}\nn\\
	&=\frac{1}{n^2} \#\big\{(x,y)\in V(G_n)\times V(G_n)\colon |\cluster(x)|, |\cluster(y)|\geq k, x\nc y\big\}.
	}
By the Markov inequality,
	\eqan{
	\lim_{k\ra\infty} \limsup_{n\rightarrow \infty}\prob(X_{n,k} \geq \vep)
	&\leq  \lim_{k\ra\infty} \limsup_{n\rightarrow \infty}\frac{1}{\vep}\expec[X_{n,k}]=0,\nn
	}
by our main assumption in \eqref{almost-local-cond}. As a result, $X_{n,k}=\opk(1)$.
We conclude that, by \eqref{Cmax-ratio},
	\eqan{
	\frac{|\Cmax|}{n}&\geq \frac{\zeta^2+\opk(1)}{\zeta+\opk(1)}=\zeta+\opk(1).
	}
Since $k$ is arbitrary, this proves that $|\Cmax|\geq n\zeta(1+\op(1))$. Further, by Corollary \ref{cor-giant-UB}, also $|\Cmax|\leq n\zeta(1+\op(1))$. Therefore, $|\Cmax|/n\convp \zeta.$ 

In turn, since  $|\Cmax|/n\convp \zeta$ and by \eqref{Zgeqk-once-more},
	\eqn{
	\frac{1}{n}|\cluster_{\sss(2)}|\indic{|\cluster_{\sss(2)}|\geq k}\leq \frac{1}{n} \sum_{i\geq 2} |\cluster_{\sss(i)}|\indic{|\cluster_{\sss(i)}|\geq k}=\frac{1}{n} \sum_{i\geq 1} |\cluster_{\sss(i)}|\indic{|\cluster_{\sss(i)}|\geq k}-\frac{|\Cmax|}{n}=\opk(1),
	}
which, again since $k$ is arbitrary, implies that $|\cluster_{\sss(2)}|/n\convp 0$.
\qed

\begin{remark}[Proof of necessity of \eqref{almost-local-cond}]
\label{rem-nec-suf} 
{\rm Condition \eqref{almost-local-cond} is also {\em necessary} for $|\Cmax|/n\convp \zeta$ to hold. Indeed, \eqref{rewrite-almost-local-cond} implies that when \eqref{almost-local-cond} fails, there exists a $\kappa>0$ such that
	\eqn{
	\label{almost-local-cond-fails}
	\limsup_{k\rightarrow \infty} \limsup_{n\rightarrow \infty} \expec[X_{n,k}]=\kappa>0.
	}
Then, there exists a subsequence $(n_l)_{l\geq 1}$ for which $\lim_{k\rightarrow \infty}\lim_{l\rightarrow \infty} \expec[X_{n_l,k}]=\kappa$, and, by \eqref{Zgeqk-once-more} and \eqref{Xnk-def},
	\eqan{
	\lim_{l\rightarrow \infty} \frac{1}{n_l^2}\expec[|\Cmax|^2\indic{|\Cmax|\geq k}]
	&\leq \lim_{l\rightarrow \infty} \frac{1}{n_l^2}\expec\Big[\sum_{i}|\cluster_{\sss(i)}|^2\indic{|\cluster_{\sss(i)}|\geq k}\Big]\\
	&=\lim_{l\rightarrow \infty} \frac{1}{n_l^2}\expec[Z_{\sss\geq k}^2-X_{n_l,k}]=\zeta_{\sss\geq k}^2-\kappa,\nn
	}
where $\expec[Z_{\sss\geq k}^2/n_l^2]\rightarrow \zeta_{\sss\geq k}^2$ follows from \eqref{Zgeqk-once-more} and bounded convergence.

Obviously, the same asymptotic bound holds for $\expec[|\Cmax|^2]/n_l^2$ instead of $\expec[|\Cmax|^2\indic{|\Cmax|\geq k}]$, since 
	\[
	\frac{1}{n_l^2}\expec[|\Cmax|^2]\leq \frac{1}{n_l^2}\expec[|\Cmax|^2\indic{|\Cmax|\geq k}] +\frac{k^2}{n_l^2}.
	\]
Letting $k\rightarrow \infty$ and using that $\zeta_{\sss\geq k}\rightarrow \zeta$, we obtain
	\eqn{
	\label{Cmax-second-moment-UB}
	\lim_{l\rightarrow \infty} \frac{1}{n_l^2}\expec[|\Cmax|^2]\leq \zeta^2-\kappa.
	}
By \eqref{Cmax-second-moment-UB}, we conclude that $|\Cmax|/n\convp \zeta$ cannot hold, as by bounded convergence, this would imply that also $\expec[|\Cmax|^2]/n^2\rightarrow \zeta^2$.
}\hfill $\ensymboldefinition$
\end{remark}

\subsection{Local properties of the giant}
\label{sec-prop-giant}
We next extend Theorem \ref{thm-giant-conv-LWC} by investigating the structure of the giant. For this, we first let $v_k(\Cmax)$ denote the number of vertices with degree $k$ in the giant component, and $|E(\Cmax)|$ the number of edges in the giant component. Further, for a graph $G$ and $v\in V(G)$, we write $d^{\sss (G)}_v$ for the degree of $v$ in $G$:

\begin{theorem}[Properties of the giant]
\label{thm-giant-conv-LWC-structure}
Under the assumptions of Theorem \ref{thm-giant-conv-LWC}, when $\zeta= \mu(|\cluster(\vertex)|=\infty)>0$,
	\eqn{
	\label{vell-giant-LWC}
	\frac{v_\ell(\Cmax)}{n}\convp \mu(|\cluster(\vertex)|=\infty, d_\vertex^{\sss (G)}=\ell).
	}
Further, assume that $D_n=d_{\vertex_n}^{\sss(G_n)}$ is uniformly integrable. Then,
	\eqn{
	\label{edge-giant-LWC}
	\frac{|E(\Cmax)|}{n}\convp \tfrac{1}{2} \expec_\mu\Big[d_\vertex^{\sss (G)}\indic{|\cluster(\vertex)|=\infty}\Big].
	}
\end{theorem}

\proof We now define, for $k\geq 1$ and $A\subseteq \N$ and with $d_v^{\sss(G_n)}$ the degree of $v$ in $G_n$,
	\eqn{
	\label{ZAk-def}
	Z_{\sss A, \geq k}=\sum_{v\in [n]} \indic{|\cluster(v)|\geq k, d_v^{\sss(G_n)}\in A}.
	}
Assume that $G_n$ converges locally in probability to $(G,\vertex)\sim \mu$. Then, as in \eqref{Zgeqk-once-more},
	\eqn{
	\label{ZAgeqk-conv}
	\frac{Z_{\sss A,\geq k}}{n}\convp \mu(|\cluster(\vertex)|\geq k, d_{\vertex}^{\sss(G)}\in A).
	}
Since $|\Cmax|\geq k$ whp by Theorem \ref{thm-giant-conv-LWC}, we thus obtain, for every $A\subseteq \N$, on the high-probability event $\{|\Cmax|\geq k\},$
	\eqn{
	\frac{1}{n}\sum_{a\in A} v_a(\Cmax)\leq \frac{Z_{\sss A,\geq k}}{n}\convp \mu(|\cluster(\vertex)|\geq k, d_{\vertex}^{\sss(G)}\in A).
	}
Applying this to $A=\{\ell\}^c$, we obtain that, for all $\vep>0$,
	\eqn{
	\label{lim-prob-A-degrees}
	\lim_{n\rightarrow \infty} \prob\Big(\frac{1}{n}\Big[|\Cmax|-v_\ell(\Cmax)\Big]\leq \mu(|\cluster(\vertex)|\geq k, d_{\vertex}^{\sss(G)}\neq \ell)+\vep/2\Big)=1.
	}
We argue by contradiction. Suppose that, for some $\ell$,
	\eqn{
	\label{liminf-prob-too-small}
	\limsup_{n\rightarrow \infty} \prob\Big(\frac{v_\ell(\Cmax)}{n}\leq \mu(|\cluster(\vertex)|=\infty, d_\vertex^{\sss(G)}=\ell)-\vep\Big)=\kappa>0.
	}
Then, along the subsequence $(n_l)_{l\geq 1}$ that attains the limsup in \eqref{liminf-prob-too-small}, with asymptotic probability $\kappa>0$, and using \eqref{lim-prob-A-degrees} and \eqref{liminf-prob-too-small},
	\eqn{
	\frac{|\Cmax|}{n}=\frac{1}{n}[|\Cmax|-v_\ell(\Cmax)]+\frac{v_\ell(\Cmax)}{n}\leq \mu(|\cluster(\vertex)|=\infty) -\vep/2,
	}
which contradicts Theorem \ref{thm-giant-conv-LWC}. We conclude that \eqref{liminf-prob-too-small} cannot hold, so that \eqref{vell-giant-LWC} follows. 

For \eqref{edge-giant-LWC}, we note that
	\eqn{
	|E(\Cmax)|=\tfrac{1}{2} \sum_{\ell\geq 1} \ell v_\ell(\Cmax).
	}
We divide by $n$ and split the sum over $\ell$ in $\ell\in [K]=\{1, \ldots, K\}$ and $\ell>K$ as
	\eqn{
	\label{split-edges-LWC}
	\frac{|E(\Cmax)|}{n}=\frac{1}{2n} \sum_{\ell\in [K]} \ell v_\ell(\Cmax)+\frac{1}{2n} \sum_{\ell>K} \ell v_\ell(\Cmax).
	}
For the first term in \eqref{split-edges-LWC}, by \eqref{vell-giant-LWC},
	\eqan{
	\tfrac{1}{2n} \sum_{\ell\in [K]} \ell v_\ell(\Cmax)
	&\convp \tfrac{1}{2} \sum_{\ell\in [K]} \ell \mu(|\cluster(\vertex)|=\infty, d_\vertex^{\sss(G)}=\ell)\\
	&=\tfrac{1}{2} \expec_\mu\Big[d_\vertex^{\sss(G)}\indic{|\cluster(\vertex)|=\infty, d_\vertex^{\sss(G)}\in [K]}\Big].\nn
	}
For the second term in \eqref{split-edges-LWC}, we bound, with $n_\ell$ the number of vertices in $G_n$ of degree $\ell$,
	\eqn{
	\frac{1}{2n} \sum_{\ell>K} \ell v_\ell(\Cmax)
	\leq \frac{1}{2} \sum_{\ell>K} \ell \frac{n_\ell}{n}=\frac{1}{2} \expec\big[d_{\vertex_n}^{\sss(G_n)}\indic{d_{\vertex_n}^{\sss(G_n)}>K}\mid G_n\big].
	}
By uniform integrability of $(d_{\vertex_n}^{\sss(G_n)})_{n\geq 1}$, 
	\eqn{
	\lim_{K\rightarrow \infty}\limsup_{n\rightarrow \infty}\expec\big[d_{\vertex_n}^{\sss(G_n)}\indic{d_{\vertex_n}^{\sss(G_n)}>K}\big]=0.
	}
As a result, by the Markov inequality and for every $\vep>0$, there exists a $K=K(\vep)<\infty$ such that 
	\eqn{
	\prob\Big(\expec\big[d_{\vertex_n}^{\sss(G_n)}\indic{d_{\vertex_n}^{\sss(G_n)}>K}\mid G_n\big]>\vep\Big)\rightarrow 0.
	}
This completes the proof of \eqref{edge-giant-LWC}.
\qed
\medskip

It is not hard to extend the above analysis to the local convergence in probability of the giant, as well as its complement, as formulated in the following theorem:

\begin{theorem}[Local limit of the giant]
\label{thm-giant-LWL}
Under the assumptions of Theorem \ref{thm-giant-conv-LWC}, when $\zeta= \mu(|\cluster(\vertex)|=\infty)>0$,
	\eqn{
	\label{giant-LWL}
	\frac{1}{n}\sum_{v\in \Cmax} \indic{B_r^{\sss(G_n)}(v)\simeq H_\star} \convp \mu(|\cluster(\vertex)|=\infty, B_r^{\sss(G)}(\vertex)\simeq H_\star),
	}
and
	\eqn{
	\label{non-giant-LWL}
	\frac{1}{n}\sum_{v\not\in \Cmax} \indic{B_r^{\sss(G_n)}(v)\simeq H_\star} \convp \mu(|\cluster(\vertex)|<\infty, B_r^{\sss(G)}(\vertex)\simeq H_\star).
	}
\end{theorem}

\proof The convergence in \eqref{non-giant-LWL} follows from that in \eqref{giant-LWL} combined with the fact that, by assumption,
	\eqn{
	\frac{1}{n}\sum_{v\in V(G_n)} \indic{B_r^{\sss(G_n)}(v)\simeq H_\star} \convp \mu(B_r^{\sss(G)}(\vertex)\simeq H_\star).
	}
The convergence in \eqref{giant-LWL} can be proved as for Theorem \ref{thm-giant-conv-LWC-structure}, now using that
	\eqan{
	\frac{1}{n}Z_{\sss \mathscr{H}_\star, \geq k}&\equiv\frac{1}{n}\sum_{v\in V(G_n)} \indic{|\cluster(v)|\geq k, B_r^{\sss(G_n)}(v)\in \mathscr{H}_\star}\\
	&\convp \mu(|\cluster(\vertex)|\geq k, B_r^{\sss(G)}(\vertex)\in \mathscr{H}_\star),\nn
	}
and, since $|\Cmax|/n\convp \zeta>0$ by Theorem \ref{thm-giant-conv-LWC}, on the high-probability event $\{|\Cmax|\geq k\},$
	\eqn{
	\frac{1}{n}\sum_{v\in \Cmax} \indic{B_r^{\sss(G_n)}(v)\in \mathscr{H}_\star} \leq \frac{1}{n}Z_{\sss \mathscr{H}_\star, \geq k}.
	}
We leave the details to the reader.
\qed

\subsection{Examples through cluster tail bounds outside giant}
\label{sec-second-largest}
In this section, we give examples of how we can apply our results to obtain a law of large numbers for the giant, given a bound on the tail probabilities of connected components outside the giant. These examples will follow from the following corollary:

\begin{corollary}[Law of large numbers giant given bound components outside it]
\label{cor-LLN-outside}
Let $(G_n)_{n\geq 1}$ be a sequence of graphs having size $|V(G_n)|=n$. %Let $(G,\vertex)$ be a random variable on $\mathscr{G}_\star$ having law $\mu$.
Assume that $G_n$ converges locally in probability to $(G,\vertex)\sim \mu$. Write $\zeta=\mu(|\cluster(\vertex)|=\infty)$ for the survival probability of the limiting graph $(G,\vertex)$. Assume further that
	\eqn{
	\label{bound-outside}
	\lim_{k\rightarrow \infty} \limsup_{n\rightarrow \infty}\prob(|\cluster(o_n)|\geq k, o_n\not\in \Cmax)=0.
	}
Then, as $n\rightarrow \infty$,
	\eqn{
	|\Cmax|/n\convp \zeta.
	}
Further, Theorem \ref{thm-giant-LWL} and \eqref{vell-giant-LWC} in Theorem \ref{thm-giant-conv-LWC-structure} also hold, while \eqref{edge-giant-LWC} in Theorem \ref{thm-giant-conv-LWC-structure} holds when $D_n=d_{\vertex_n}^{\sss(G_n)}$ is uniformly integrable.
\end{corollary}

\proof This is immediate, since 
	\eqan{
	&\frac{1}{n^2}\expec\Big[\#\big\{(x,y)\in V(G_n)\times V(G_n)\colon |\cluster(x)|, |\cluster(y)|\geq k, x\nc y\big\}\Big]\\
	&\qquad \leq 2\prob(|\cluster(o_n)|\geq k, o_n\not\in \Cmax),\nn
	}
since either $x$ or $y$ is not in $\Cmax$ when $x\nc y$.
\qed

\paragraph{Application to spatial inhomogeneous random graphs.}
\label{rem-examples-outside}
Our results apply to various {\em spatial inhomogeneous random graph models}, as a by-product of the work Jorritsma, Komj\'athy and Mitsche \cite{JorKomMit25} on component tails in spatial inhomogeneous random graphs (see also \cite{JorKomMit24} for related results on long-range percolation). This work, combined with the local convergence in \cite{HofHooMai21}, shows that the survival probability of the local limit indeed is the limit of the proportion of vertices in the giant, where we emphasise that the local limit is {\em not} a tree. The identification of the local limit of the giant in Theorem \ref{thm-giant-LWL} then follows immediately, as it was done in \cite[Theorem 25]{BolRio15} for the configuration model, and applies for example also to spatial inhomogeneous random graphs as studied in \cite{JorKomMit25}.
\smallskip

Jorritsma, Komj\'athy and Mitsche \cite{JorKomMit25} prove \eqref{bound-outside} (as well as several related results) for a wide collection of models. These models include {\em geometric inhomogeneous random graphs} (GIRGs), as defined in \cite{BriKeuLen16}, in the case where the degree power-law exponent $\tau$ satisfies $\tau\in(2,3)$. Further, they include {\em hyperbolic random graphs} with degree power-law exponent $\tau\in(2,3)$, as defined in \cite{KriPapKitVahBog10} and studied further in \cite{FouMul18, KomLod20,HofHooMai21}, using the fact that they can be mapped to one-dimensional GIRGs. The law of large number of the giant was already proved in \cite{FouMul18}, the relation to the local limit obviously follows from our work. Let us next formally state the convergence of the giant result.
\smallskip

First, we define the class of GIRGs to which these results apply. We denote the graph by $G_n=(V(G_n),E(G_n))$, and let the vertices $V(G_n)$ of $G_n$ be the points of a Poisson point process on $\tfrac{1}{2}[-n^{1/d}, n^{1/d}]\times [1,\infty)$ with intensity measure
	\eqn{
	\label{mu-i-KSRG}
	\mu({\mathrm d}x, {\mathrm d}w)={\mathrm d}x \times (\tau-1)w^{-\tau} {\mathrm d}w.
	}
For $p\geq0$, $\alpha\in [0,\infty]$, $\sigma\in[0,\infty)$, and $\beta>0$, we let
	\eqn{
	\label{p-i-KSRG}
	p((x,w), (x',w'))=\begin{cases}
				p\min\Big\{1, \Big(\frac{\beta (x\wedge x')(x\vee x')^{\sigma}}{\|x-x'\|}\Big)^\alpha\Big\} &\text{for }\alpha<\infty,\\
				p\indicwo{\big\{\beta (x\wedge x')(x\vee x')^{\sigma}\geq \|x-x'\|\big\}} &\text{for }\alpha=\infty.
				\end{cases}
	}
For $u,v\in V(G_n)$, and given the weights $(w_v)_{v\in V(G_n)}$, we independently let $\{u,v\}\in E(G_n)$ with probability $p((u,w_u),(v,w_v))$, and $\{u,v\}\not\in E(G_n)$ otherwise.

\begin{theorem}[Giant in GIRGs]
\label{thm-giant-GIRG}
Let $G_n$ be the GIRG defined in \eqref{mu-i-KSRG} and \eqref{p-i-KSRG} with $p\geq0$, $\alpha\in [1,\infty]$, $\sigma>0$, and $\beta>0$. Assume that $\tau\in(2,2+\sigma)$. Then $|\Cmax|/n\convp \zeta$ whp, where $\zeta$ is the survival probability of the Poisson infinite GIRG with edge probabilities given by \eqref{mu-i-KSRG} and \eqref{p-i-KSRG}, while $|\cluster_{\sss(2)}|/n\convp 0$.\footnote{This result also appeared in \cite[Theorem 9.38]{Hofs24}. Unfortunately, \cite[Theorem 9.38]{Hofs24} contains incorrect conditions on the parameters, in that the result was claimed for $\alpha\in[0,\infty)$ rather than $\alpha\in[1,\infty)$, $\sigma\in[0,\infty)$ rather than $\sigma>0$, and the condition $\tau\in (2,2+\sigma)$ was missing.}
\end{theorem}

\proof The model we describe is exactly the same is that studied in \cite[(1.5)--(1.7)]{JorKomMit25}. Indeed, due to the construction in \eqref{mu-i-KSRG}, the vertex locations form a homogeneous Poisson point process, while, conditionally on the vertex locations, the weights are iid random variables with density $f_{\sss W}(w)=(\tau-1)w^{-\tau}$ for $w\geq 1$, as in \cite[Assumption 1.3]{JorKomMit25}. Local convergence was proved in \cite{HofHooMai21} (in fact, under much weaker conditions). The condition \eqref{bound-outside} follows from \cite[Theorem 2.1]{JorKomMit25}, so that Theorem \ref{thm-giant-GIRG} follows from Corollary \ref{cor-LLN-outside}.
\qed

%\hfill $\ensymboldefinition$

\begin{remark}[Relation to second largest component]
\label{rem-examples}
{\rm There are many examples for which linear-size {\em lower bounds} on the giant have been proved, jointly with a sublinear upper bound on the second component. It would be of interest to investigate how a bound on the second largest component can be leveraged to extend our law of large numbers for the giant to this setting. While a bound on the {\em size} of the second largest intuitively seems a huge step, it is not difficult to convince oneself that this does not suffice, as there can be many connected components that are large, but do not contribute to the giant. There are quite a few sublinear bounds on the second largest for graphs whose local limit was identified. For example, for the GIRG studied in \cite{BriKeuLen16}, under appropriate conditions, the local limit was identified in \cite{HofHooMai21}, and the upper bound on the second largest component proved in \cite{BriKeuLen16}. For the hyperbolic random graph, again the local limit was identified in \cite{HofHooMai21}, the lower bound on the giant proved in \cite{BodFouMul15}, and the bound on the second largest component in \cite{KiwMit19}.
}\hfill $\ensymboldefinition$
\end{remark}

\subsection{The `giant is almost local' condition revisited}
\label{sec-GIOL-rev}
The `giant is almost local' condition \eqref{almost-local-cond} is sometimes not so convenient to verify directly, and we now give an alternative form that is often easier to work with. In its statement, we write $\partial B^{\sss(G)}_r(v)$ for the collection of vertices at graph distance equal to $r$ from $v$:

\begin{lemma}[Condition \eqref{almost-local-cond} revisited]
\label{lem-almost-local-cond-rep}
Consider $(G_n)_{n\geq 1}$ under the conditions of Theorem \ref{thm-giant-conv-LWC}. Assume further that there exists $r=r_k\rightarrow \infty$ such that
	\eqn{
	\label{cond-limit}
	\mu(|\cluster(\vertex)|\geq k, |\partial B^{\sss(G)}_r(\vertex)|<r_k)\rightarrow 0.
	%\qquad
	%\mu(|\cluster(\vertex)|<k, |\partial B_r^{\sss(G)}(\vertex)|\geq r_k)\rightarrow 0.
	}
Then, the `giant is almost local' condition in \eqref{almost-local-cond} holds when
	\eqn{
	\label{almost-local-cond-rep}
	\lim_{r\ra\infty} \limsup_{n\rightarrow \infty}\frac{1}{n^2}\expec\Big[\#\big\{(x,y)\in V(G_n)\times V(G_n)\colon |\partial B^{\sss(G_n)}_r(x)|, |\partial B^{\sss(G_n)}_r(y)|\geq r, x\nc y\big\}\Big]=0.
	}

\end{lemma}

\proof Denote 
	\eqan{
	\label{error-terms-simplification}
	P_k&=\#\big\{(x,y)\in V(G_n)\times V(G_n)\colon |\cluster(x)|, |\cluster(y)|\geq k, x\nc y\big\},\\
	P_r^{\sss(2)}&=\#\big\{(x,y)\in V(G_n)\times V(G_n)\colon |\partial B_r^{\sss(G_n)}(x)|, |\partial B_r^{\sss(G_n)}(y)|\geq r, x\nc y\big\}.
	}
Then,
	\eqn{
	P_k
	\leq P_r^{\sss(2)}+ 2n Z_{\sss <r, \geq k},
	}
where
	\eqan{
	Z_{\sss <r, \geq k}&=\sum_{v\in V(G_n)} \indic{|\partial B_r^{\sss(G_n)}(v)|<r, |\cluster(v)|\geq k}.
	%Z_{\sss \geq r, < k}=\sum_{v\in V(G_n)} \indic{|\partial B_r^{\sss(G_n)}(v)|\geq r, |\cluster(v)|<k}.
	}
Therefore,	 by local convergence in probability,
	\eqan{
	\label{diff-conditions}
	\frac{1}{n^2}[P_k-P_r^{\sss(2)}]
	&\leq \frac{2}{n}Z_{\sss <r, \geq k}
	\convp 2\mu(|\cluster(\vertex)|\geq k, |\partial B_r^{\sss(G)}(\vertex)|<r).
	}
Take $r=r_k$ as in \eqref{cond-limit}, so that the rhs of \eqref{diff-conditions} vanishes, and, by Dominated Convergence, also 
	\eqn{
	\lim_{k\ra\infty} \limsup_{n\rightarrow \infty}\frac{1}{n^2}\expec\big[P_{k}-P_{r_k}^{\sss(2)}\big]\leq 0.
	}
We arrive at
	\eqn{
	\lim_{k\ra\infty} \limsup_{n\rightarrow \infty}\frac{1}{n^2}\expec\big[P_k\big]
	\leq \lim_{k\ra\infty} \limsup_{n\rightarrow \infty}\frac{1}{n^2}\expec\big[P_{r_k}^{\sss(2)}\big]
	=0,
	}
by \eqref{almost-local-cond-rep} and since $r_k\rightarrow \infty$ when $k\rightarrow \infty$.
\qed
\medskip

\begin{remark}[Alternative `giant is almost local' condition]
\label{rem-alt-GiaL-cond}
{\rm The assumption in \eqref{almost-local-cond-rep} is sometimes more convenient than \eqref{almost-local-cond}, as it requires that most pairs of vertices with many vertices at distance $r$ are connected to one another. In many random graphs, there is some weak dependence between $B_r^{\sss(G_n)}(x)$ and the graph outside of it. This is more complicated when dealing with $|\cluster(x)|\geq k$. 

The assumption in \eqref{cond-limit} on the local limit is often easily verified. For example, for the unimodular branching processes with bounded offspring, to which we will apply it below, we can take $r_k=k$ and use that, on the event of survival,
	$
	\nu^{-r} |\partial B_r^{\sss(G)}(\vertex)|\convas M,
	$
where $(G,\vertex)$ denotes the unimodular branching process with bounded offspring, $M>0$ on the event of survival by \cite[Theorem 3.9]{Hofs17}, and $\nu$ the expected offspring of this branching process. Therefore, $\mu(|\cluster(\vertex)|\geq k, |\partial B_k^{\sss(G)}(\vertex)|<k)\rightarrow 0$ as $k\rightarrow \infty$. However, there are examples where \eqref{cond-limit} fails, while \eqref{almost-local-cond} holds, so that \eqref{almost-local-cond} and \eqref{almost-local-cond-rep} are not equivalent.}\hfill $\ensymboldefinition$
\end{remark}

\section{Application to the configuration model}
\label{sec-giant-CM}
In this section, we apply our results to the configuration model as introduced by Bollob\'as \cite{Boll80b} in the context of random regular graphs. %An application to \erdos{} is provided in \cite[Chapter 2]{Hofs24}. 
The giant in the configuration model has a long history. It was first investigated by Molloy and Reed \cite{MolRee95, MolRee98} in a setting where the degrees are general. The problem was revisited by Janson and Luczak \cite{JanLuc09} and Bollob\'as and Riordan \cite{BolRio15}, amongst others. This section is organised as follows. We start by introducing the model and stating our results in Section \ref{sec-model-def-res}. We then prove the main result in Section \ref{sec-giant-almost-local-CM}, and state a consequence on typical graph distances, that is proved along the way, in Section \ref{sec-small-world}. Some proofs of technical ingredients are deferred to Appendix \ref{app-A}.

\subsection{Model definition and results}
\label{sec-model-def-res}
The configuration model has the nice property that, when conditioned on being simple, it yields a uniform random graph with the prescribed degree distribution. We refer to \citeI{Chapter 7} for an extensive introduction.

\paragraph{\bf Model definition and assumptions.}
Fix an integer $n$ that will denote the number of vertices in the random graph. Consider a sequence of degrees $\bfdit=(d_i)_{i\in[n]}$. %The aim is to construct an undirected (multi)graph with $n$ vertices, where vertex $j$ has degree $d_j$. Here a multi-graph is a graph {\it possibly} having self-loops and multiple edges between pairs of vertices. 
Without loss of generality, we assume throughout this paper that $d_j\geq 1$ for all $j\in [n]$, since when $d_j=0$, vertex $j$ is isolated and can be removed from the graph. We assume that the total degree
    	$
    	\ell_n=\sum_{j\in [n]} d_j
    	$
is even.
\smallskip
%We wish to construct a simple graph such that $\bfdit=(d_i)_{i\in[n]}$ are the degrees of the $n$ vertices. However, even when $\ell_n=\sum_{j\in [n]} d_j$ {\it is} even, this is not always possible.
%\medskip

%Since it is not always possible to construct a simple graph with a given degree sequence, instead, we construct a {\it multi-graph}. One way of obtaining such a multi-graph with the given degree sequence is to pair the half-edges attached to the different vertices in a uniform way. Two half-edges together form an edge, thus creating the edges in the graph. Let us explain this in more detail.
%\medskip

To construct the multi-graph where vertex $j$ has degree $d_j$ for all $j\in [n]$, we have $n$ separate vertices and incident to vertex $j$, we have $d_j$ half-edges. %Every half-edge needs to be connected to another half-edge to form an edge, and by forming all edges we build the graph. 
We number the half-edges in an arbitrary order from $1$ to $\ell_n$, and start by randomly connecting the first half-edge with one of the $\ell_n-1$ remaining half-edges. Once paired, two half-edges form a single edge of the multi-graph, and the half-edges are removed from the list of half-edges that need to be paired. %Hence, a half-edge can be seen as the left or the right half of an edge. 
We continue the procedure of randomly choosing and pairing the half-edges until all half-edges are connected, and call the resulting graph the {\it configuration model with degree sequence $\bfdit$}, abbreviated as $\CMnd$. A careful reader may worry about the order in which the half-edges are being paired. In fact, this ordering turns out to be irrelevant since the random pairing of half-edges is completely {\em exchangeable}. It can even be done in a {\em random} fashion, which will be useful when investigating {\em neighbourhoods} in the configuration model. See e.g., \citeI{Definition 7.5 and Lemma 7.6} for more details.

We denote the degree of a uniformly chosen vertex $\Ver$ in $[n]$ by $D_n=d_{\Ver}$. The random variable $D_n$ has distribution function $F_n$ given by
    \eqn{
    \label{def-Fn-CM}
    F_n(x)=\frac{1}{n} \sum_{j\in [n]} \indic{d_j\leq x},
    }
which is the {\em empirical distribution of the degrees.} Equivalently, $\prob(D_n=k)=n_k/n$, where $n_k$ denotes the number of vertices of degree $k$.
We assume that the vertex degrees satisfy the following \emph{regularity conditions:}

\begin{cond}[Regularity conditions for vertex degrees]
\label{cond-degrees-regcondII}
~\\
{\bf (a) Weak convergence of vertex weight.}
There exists a distribution function $F$ such that, as $n\rightarrow \infty$,
    \eqn{
    \label{Dn-weak-convII}
    D_n\convd D,
    }
where $D_n$ and $D$ have distribution functions $F_n$ and $F$, respectively. 
%Equivalently, for any $x$, 
%    \eqn{
%    \label{conv-Fn-CMII}
%    \lim_{n\rightarrow \infty} F_n(x)=F(x).
%    }
Further, we assume that $F(0)=0$, i.e., $\prob(D\geq 1)=1$.\\
{\bf (b) Convergence of average vertex degrees.} As $n\rightarrow \infty$,
    \eqn{
    \label{conv-mom-DnII}
    \expec[D_n]\rightarrow \expec[D]<\infty,
    }
where $D_n$ and $D$ have distribution functions $F_n$ and $F$ from part (a), respectively.
%\\
%{\bf (c) Convergence of second moment vertex degrees.} As $n\rightarrow \infty$,
%    \eqn{
%    \label{conv-sec-mom-DnII}
%    \expec[D_n^2]\rightarrow \expec[D^2]<\infty.
%    }
%where $D_n$ and $D$ have distribution functions $F_n$ and $F$ from part (a), respectively.
\end{cond}
%\medskip
Note that Conditions \ref{cond-degrees-regcondII}(a)-(b) are equivalent to Condition \ref{cond-degrees-regcondII}(a) and uniform integrability of $(D_n)_{n\geq 1}$. While Conditions \ref{cond-degrees-regcondII}(a)-(b) may appear to be quite strong, in fact, they are quite close to assuming uniform integrability of $(D_n)_{n\geq 1}$. Indeed, when $(D_n)_{n\geq 1}$ is uniformly integrable, then there is a subsequence along which $D_n$ converges in distribution as in Condition \ref{cond-degrees-regcondII}(a), and along this subsequence also Condition \ref{cond-degrees-regcondII}(b) holds. We can then apply our results along this subsequence. 

\paragraph{\bf The giant in the configuration model.}
We now come to our result on the connected components in the configuration model $\CMnd$, where the degrees $\bfdit=(d_i)_{i\in[n]}$ satisfy Conditions \ref{cond-degrees-regcondII}(a)-(b): %We write $p_k=\prob(D=k)$ for the asymptotic degree distribution. 
%We recall that $D_n$ is the degree of a uniformly chosen vertex in $[n]$, i.e., $D_n=d_{\Ver}$, where $\Ver$ is uniformly chosen from $[n]$. Equivalently,
%    \eqn{
%    \label{JanLuc09_1}
%    \prob(D_n=k)=n_k/n,
%    }
%where $n_k$ denotes the number of vertices of degree $k$. 
%Further, for a graph $G$, we write $v_k(G)$ for the number of vertices of degree $k$ in $G$, and $|E(G)|$ for the number of edges. The main result concerning the size and structure of the largest connected components of $\CMnd$ is the following:

\begin{theorem}[Phase transition in $\CMnd$]
\label{thm-convtwolargestcomps}
Suppose that Conditions \ref{cond-degrees-regcondII}(a)-(b) hold and consider the random graph $\CMnd$, letting $n\to\infty$. Assume that $p_2=\prob(D=2)<1$. %Let $\Cmax$ and $\cluster_{\sss(2)}$ be the largest and second largest components of $\CMnd$ (breaking ties arbitrarily).
\ben
\item[(a)] If $\nu=\expec[D(D-1)]/\expec[D]>1$, then there exist
$\xi\in [0,1),\zeta \in(0,1]$ such that
	\begin{eqnarray*}
	|\Cmax|/n&\convp& \zeta\text{,}\\
	v_k(\Cmax)/n&\convp& \prob(D=k)(1-\xi^k)\quad\text{ for every }k\geq 0\text{,}\\
	|E(\Cmax)|/n&\convp& \tfrac{1}{2}\expec[D] (1-\xi^2).
	\end{eqnarray*}
while $|\cluster_{\sss(2)}|/n \convp 0$ and $|E(\cluster_{\sss(2)})|/n\convp 0$.

\item[(b)] If $\nu=\expec[D(D-1)]/\expec[D]\leq 1$, then $|\Cmax|/n\convp0$ and $|E(\Cmax)|/n\convp0$.
\een
%Consequently, the same result holds for the uniform random graph with degree sequence $\bfdit$ satisfying Condition \ref{cond-degrees-regcondII}(a)-(b), under the extra assumption that $\sum_{i\in[n]}d_i^2=O(n)$.
\end{theorem}

We prove Theorem \ref{thm-convtwolargestcomps} in Section \ref{sec-giant-almost-local-CM} below. We now remark upon the result and on the conditions arising in it.

\paragraph{\bf Local structure of $\Cmax$.}
Theorem \ref{thm-convtwolargestcomps} is the main result for $\CMnd$ proved in \cite{JanLuc09}. The result in \cite{BolRio15} only concerns $|\Cmax|/n$. Related results are in \cite{MolRee95,MolRee98,Rior12}. Many of these proofs use an {\em exploration} of the connected components in the graph, %albeit that the precise explorations can be performed in rather different ways in that 
in \cite{MolRee95,MolRee98} in discrete time and in \cite{JanLuc09} in continuous time. Further, \cite{BolRio15} relies on concentration inequalities combined with a sprinkling argument. For $\CMnd$, \cite[Theorem 25]{BolRio15} identifies the local structure of the giant that follows from Theorem \ref{thm-giant-LWL}. \cite[Theorem 25]{BolRio15} also applies to uniform simple graphs with given degrees.
%The present approach is, even in this light, novel, since it also allows one to study the local structure of the giant and non-giant, as stated in Theorem \ref{thm-giant-LWL}. It is not clear whether the proof in \cite{JanLuc09}, for example, can be easily extended to include this extra information. 

\paragraph{\bf Local convergence of configuration models.}
The fact that the configuration model converges locally in probability is well-established, and is the starting point for the proof of Theorem \ref{thm-convtwolargestcomps} using Theorems \ref{thm-giant-conv-LWC} and \ref{thm-giant-conv-LWC-structure}. We refer to Appendix \ref{app-LC-CM} for more details on local convergence of $\CMnd$. Dembo and Montanari \cite{DemMon10a} crucially rely on it in order to identify the limiting pressure for the Ising model on the configuration model. Many alternative proofs exist. We use some of the ingredients in \cite[Proof of Theorem 4.1]{Hofs24}, since these are helpful in the proof of Theorem \ref{thm-convtwolargestcomps} as well. We also refer to the lecture notes by Bordenave \cite{Bord16} for a nice exposition of local convergence proofs for the configuration model. Further, Bordenave and Caputo \cite{BorCap15} prove that the neighbourhoods in the configuration model satisfy a large deviation principle at speed $n$ in the context where the degrees of the configuration model are bounded. 
\smallskip

Let us now describe the local limit of $\CMnd$ subject to Conditions \ref{cond-degrees-regcondII}(a)-(b). The root has offspring distribution $(p_k)_{k\geq 1}$ where $p_k=\prob(D=k)$ and $D$ is from Condition \ref{cond-degrees-regcondII}(a), while all other individuals in the tree have offspring distribution $(p_k^\star)_{k\geq 0}$ given by
	\eqn{
	\label{pk-star-def}
	p_k^\star =\frac{k+1}{\expec[D]}\prob(D=k+1).
	}
The distribution $(p_k^\star)_{k\geq 0}$ has the interpretation of the {\em forward degree} of a uniform edge. The above branching process is a so-called {\em unimodular} branching process (recall \cite{AldLyo07}) with root offspring distribution $(p_k)_{k\geq 1}$.

%Having identified the local limit of $\CMnd$, to prove Theorem \ref{thm-convtwolargestcomps}, we are left to prove the `giant component is almost local' condition in \eqref{almost-local-cond}, which we will do in the form of \eqref{almost-local-cond-rep}.

\paragraph{\bf Reformulation in terms of branching processes.}
We next interpret the results in Theorem \ref{thm-convtwolargestcomps} in terms of our unimodular branching processes. In terms of this, $\xi$ is the extinction probability of a branching process with offspring distribution $(p_k^\star)_{k\geq 0}$, and $\zeta$ is the survival probability of the unimodular branching process with root offspring distribution $(p_k)_{k\geq 1}$. Thus, $\zeta$ satisfies
    \eqn{
    \label{zeta-def-CM}
    \zeta=\sum_{k\geq 1} p_k(1-\xi^k),
    }
with $\xi$ the smallest solution to
    \eqn{
    \label{xi-def-CM}
    \xi=\sum_{k\geq 0}p_k^\star \xi^k.
    }
%Note that $\xi$ can be interpreted as the branching process extinction probability for a branching process with offspring distribution $(p_k^\star)_{k\geq 0}$. Thus, 
By branching process theory (see e.g., \cite[Theorem 3.1]{Hofs17}) $\xi=1$ precisely when $\nu\leq 1$, where
    \eqn{
    \nu=\sum_{k\geq 0} kp_k^\star 
    =\frac{1}{\expec[D]}\sum_{k\geq 0} k(k+1)p_{k+1}
    =\expec[D(D-1)]/\expec[D],
    }
by \eqref{pk-star-def}. This explains the condition on $\nu$ in Theorem \ref{thm-convtwolargestcomps}(a). Further, to understand the asymptotics of $v_k(\Cmax)$, we note that there are $n_k\approx np_k$ vertices with degree $k$. Each of the $k$ direct neighbours of a vertex of degree $k$ survives with probability close to $1-\xi$, so that the probability that at least one of them survives is close to $1-\xi^k$. When one of the neighbours of the vertex of degree $k$ survives, the vertex itself is part of the giant component, which explains why $v_k(\Cmax)/n\convp p_k(1-\xi^k)$. Finally, an edge consists of two half-edges, and an edge is part of the giant component precisely when one of the vertices incident to it is, which occurs with asymptotic probability $1-\xi^2$. There are in total $\ell_n/2=n\expec[D_n]/2\approx n \expec[D]/2$ edges, which explains why $|E(\Cmax)|/n\convp\frac{1}{2}\expec[D] (1-\xi^2)$. Therefore, all results in Theorem \ref{thm-convtwolargestcomps} have a simple explanation in terms of the branching-process approximation of the connected component for $\CMnd$ of a uniform vertex
in $[n]$.

\paragraph{\bf The condition $\prob(D=2)=p_2<1$.}
%Because isolated vertices do not matter, without loss of generality, we may assume that $p_0 = 0$. 
The case $p_2=1$, for which $\nu = 1$, is quite exceptional, and quite different limits for $|\Cmax|/n$ can occur, see \cite[Section 4.2]{Hofs24} and \cite{Fede20} for more details.

\subsection{The `giant component is almost local' proof}
\label{sec-giant-almost-local-CM}
In this section, we prove Theorem \ref{thm-convtwolargestcomps} using the `giant is almost local' results in Theorems \ref{thm-giant-conv-LWC} and \ref{thm-giant-conv-LWC-structure}. We start by setting the stage.

\paragraph{\bf Setting the stage for the proof of Theorem \ref{thm-convtwolargestcomps}.}
Theorem \ref{thm-convtwolargestcomps}(b) follows directly from Corollary \ref{cor-giant-UB} combined with the local convergence in probability discussed below Theorem \ref{thm-convtwolargestcomps} and the fact that, for $\nu\leq 1$, the unimodular branching process with root offspring distribution $(p_k)_{k\geq 0}$ given by $p_k=\prob(D=k)$ dies out.
\smallskip

Theorem \ref{thm-convtwolargestcomps}(a) follows from Theorem \ref{thm-giant-conv-LWC-structure}, together with the facts that, for the unimodular branching process with root offspring distribution $(p_k)_{k\geq 0}$ given by $p_k=\prob(D=k)$,
	\eqn{
	\mu(|\cluster(\vertex)|=\infty, d_\vertex=\ell)=p_\ell(1-\xi^\ell),
	}
and
	\eqn{
	\expec_\mu\Big[d_\vertex\indic{|\cluster(\vertex)|=\infty}\Big]=\expec[D] (1-\xi^2).
	}
Thus, it suffices to check the assumptions to Theorem \ref{thm-giant-conv-LWC-structure}. The uniform integrability of $D_n$ follows from Conditions \ref{cond-degrees-regcondII}(a)-(b).
For the conditions in Theorem \ref{thm-giant-conv-LWC}, the local convergence in probability is discussed below Theorem \ref{thm-convtwolargestcomps}, so we are left to proving the crucial hypothesis in \eqref{almost-local-cond}, which the remainder of the proof will do.
\smallskip

To start with our proof of \eqref{almost-local-cond}, applied to $\CMnd$, we first use the alternative formulation from Lemma \ref{lem-almost-local-cond-rep}, and note that \eqref{cond-limit} holds for the limiting unimodular branching process (recall Remark \ref{rem-alt-GiaL-cond}). Then, with $\vertex_1,\vertex_2\in[n]$ chosen independently and uar,
	\eqan{
	&\frac{1}{n^2}\expec\Big[\#\big\{(x,y)\in V(G_n)\colon |\partial B_r^{\sss(G_n)}(x)|, |\partial B_r^{\sss(G_n)}(y)|\geq r, x\nc y\big\}\Big]\\
	&\qquad=\prob(|\partial B_r^{\sss(G_n)}(\Ver_1)|, |\partial B_r^{\sss(G_n)}(\Ver_2)|\geq r, \Ver_1\nc \Ver_2).\nn
	}
Thus, our main aim is to show that
	\eqn{
	\label{aim-PT-CM}
	\lim_{r\rightarrow \infty} \limsup_{n\rightarrow \infty} \prob(|\partial B_r^{\sss(G_n)}(\Ver_1)|, |\partial B_r^{\sss(G_n)}(\Ver_2)|\geq r, \Ver_1\nc \Ver_2)=0.
	}
%This is what we will focus on from now on.
%\medskip

\paragraph{\bf Coupling to $n$-dependent branching process.}
We next relate the neighbourhood in a random graph to a certain $n$-dependent unimodular branching process where the root has offspring distribution $D_n$ in Condition \ref{cond-degrees-regcondII}. Such a coupling has previously appeared in \cite{BhaHofHoo17}. Since the branching process is unimodular, all other individuals have offspring distribution $D_n^\star-1$,  where 
	\eqn{
	\prob(D_n^\star=k)=\frac{k}{\expec[D_n]}\,\prob(D_n=k),\qquad k\in\N,
	}
is the size-biased distribution of $D_n$. Denote this branching process by $(\Treen(t))_{t\in\N_0}$. Here, $\Treen(t)$ denotes the branching process when precisely $t$ vertices have been explored, and we explore it in the breadth-first order. Clearly, by Conditions \ref{cond-degrees-regcondII}(a)-(b), $D_n\convd D$ and $D_n^\star\convd D^\star$, which implies that $\Treen(t)\convd \BP(t)$ for every $t$ finite, where $\BP(t)$ is the restriction of the unimodular branching process $\BP$ with root offspring distribution $(p_k)_{k\geq 1}$ for which $p_k=\prob(D=k)$ to its first $t$  individuals.

Below, we extend the coupling of the graph exploration in $\CMnd$ and $(\Treen(t))_{t\in\N_0}$ significantly. For this, we let $(\Graphn(t))_{t\in\N_0}$ denote the graph exploration process from a uniformly chosen vertex $\Ver\in[n]$. Here $\Graphn(t)$ is the graph exploration after pairing $t$ half-edges, in the breadth-first manner.  In particular, from $(\Graphn(t))_{t\in\N_0}$ we can retrieve $(B_r^{\sss(G_n)}(\vertex))_{r\in\N_0}$. The following lemma shows that we can couple the graph exploration to the unimodular branching 
process in such a way that $(\Graphn(t))_{0\leq t\leq m_n}$ is equal to $(\Treen(t))_{0\leq t\leq m_n}$ whenever $m_n\rightarrow \infty$ sufficiently slowly. In the statement, we write $(\hatGraphn(t), \hatTreen(t))_{t\in\N_0}$ for the coupling of $(\Graphn(t))_{0\leq t\leq m_n}$ and $(\Treen(t))_{0\leq t\leq m_n}$, and write $(\hatGraphn(t))_{0\leq t\leq m}\neq (\hatTreen(t))_{0\leq t\leq m}$ when a miscoupling occurs in the first $m$ generations:

\begin{lemma}[Coupling graph exploration and branching process]
\label{lem-coupling-CM}
Subject to Conditions \ref{cond-degrees-regcondII}(a)-(b), there exists a coupling $(\hatGraphn(t), \hatTreen(t))_{t\in\N_0}$ of $(\Graphn(t))_{0\leq t\leq m_n}$ and $(\Treen(t))_{0\leq t\leq m_n}$ such that
	\eqn{
	\label{Graph-BP}
	\prob\Big((\hatGraphn(t))_{0\leq t\leq m_n}\neq (\hatTreen(t))_{0\leq t\leq m_n}\Big)=o(1),
	}
when $m_n=o(\sqrt{n/\expec[D_n^2]})\rightarrow \infty$, where $\expec[D_n^2]=\frac{1}{n}\sum_{v\in [n]}d_v^2$ denotes the average of the squares of the degrees.
\end{lemma}

Lemma \ref{lem-coupling-CM} also implies that the proportion of vertices whose $r$ neighbourhood is isomorphic to a specific tree converges to the probability that the unimodular branching process with root offspring distribution $(p_k)_{k\geq 1}$ is isomorphic to this tree, which is a crucial ingredient in local convergence. Lemma \ref{lem-coupling-CM} is a more precise version of \cite[Lemma 4.2]{Hofs24}.

\begin{remark}[Conditions and extensions]
\label{rem-extensions-lemma-coupling-CM}
{\rm Note that $\expec[D_n^2]=\Theta(\dmax)$, where $\dmax$ is the maximal degree in the graph. We will apply Lemma \ref{lem-coupling-CM} in the setting where the degrees are uniformly bounded. However, it holds more generally under Conditions \ref{cond-degrees-regcondII}(a)-(b). Indeed, note that $\dmax=o(n)$ follows from Conditions \ref{cond-degrees-regcondII}(a)-(b). Further, Lemma \ref{lem-coupling-CM} can easily be extended to  explorations from {\em two} independent sources $(\Ver_1,\Ver_2)$, where we can still take $m_n=o(\sqrt{n/\expec[D_n^2]})$, and the two branching processes to which we couple the exploration from two sources, denoted by $(\hatTreen^{\sss(1)}(t))_{0\leq t\leq m_n}$ and $(\hatTreen^{\sss(2)}(t))_{0\leq t\leq m_n}$, are independent.\hfill $\ensymboldefinition$
}
\end{remark}
We refer to Appendix \ref{app-coupling-B-BP} for the proofs of Lemma \ref{lem-coupling-CM} and Remark \ref{rem-extensions-lemma-coupling-CM}.
\smallskip

\paragraph{\bf Start of proof for configuration models with bounded degrees.} We start by proving the result for configuration models with {\it bounded degrees}, i.e., for now we assume that $\max_{v\in[n]}d_v\leq b$. 
Take an arbitrary $\underline{m}_n=o(\sqrt{n})$, then Remark \ref{rem-extensions-lemma-coupling-CM} shows that whp we can {\em perfectly} couple $(B_k^{\sss(G_n)}(\Ver_1))_{k\leq \underline{k}_n}$ to a unimodular branching processes $(\BP_k^{\sss(1)})_{k\leq \underline{k}_n}$ with root offspring distribution $(\prob(D_n=k))_{k\geq 1}$. Here $|\BP_k^{\sss(1)}|$ is the size of the $k$th generation in our $n$-dependent branching process, while we define
	\eqn{
	\underline{k}_n=\inf\big\{k\colon |\BP_k^{\sss(1)}|\geq \underline{m}_n\big\}.
	}
Since all degree are bounded by $b$, we have that $|\BP_{\underline{k}_n}^{\sss(1)}|\leq b\underline{m}_n=\Theta(\underline{m}_n)$. Below, we will need that we can even extend the coupling up to generation $\underline{k}_n+1$, since $|\BP_{\underline{k}_n+1}^{\sss(1)}|\leq (b+(b-1)^2)\underline{m}_n=\Theta(\underline{m}_n)$, and then the same bounds hold for $(B_k^{\sss(G_n)}(\Ver_1))_{k\leq \underline{k}_n}$ on the event of perfect coupling. Let $\Ccal_n(1)$ denote the {\em perfect coupling event} from vertex $\Ver_1$, so that
	\eqn{
	\label{Cevent1-n-def-CM}
	\Ccal_n(1)=\big\{(|\partial B_k^{\sss(G_n)}(\Ver_1)|)_{k\leq \underline{k}_n+1}=(|\BP_k^{\sss(1)}|)_{k\leq \underline{k}_n+1}\big\},
	\quad \text{and}\quad 
	\prob(\Ccal_n(1))=1-o(1).
	}
%\medskip

%We have thus whp successfully coupled the graph exploration $(|\partial B_k^{\sss(G_n)}(\Ver_1)|)_{k\leq \underline{k}_n}$ to the $n$-dependent unimodular branching process $(\BP_k^{\sss(1)})_{k\leq \underline{k}_n}$ up to the generation $\underline{k}_n$ where their total sizes are still at most $\underline{m}_n=o(\sqrt{n})$. 
We extend the above coupling to also deal with vertex $\Ver_2$, for which we explore a little further, but do not requite perfect coupling. For this, we start by defining the necessary notation. Let $(\BP_k^{\sss(2)})_{k\geq 0}$ be an $n$-dependent unimodular branching process independent of $(\BP_k^{\sss(1)})_{k\geq 0}$. For $\overline{m}_n\geq \underline{m}_n$, we let
	\eqn{
	\bar{k}_n=\inf\big\{k\colon |\BP_k^{\sss(2)}|\geq \overline{m}_n\big\},
	}
and, again since all degrees are bounded, $|\BP_{\bar{k}_n+1}^{\sss(2)}|\leq (b+(b-1)^2)\overline{m}_n=\Theta(\overline{m}_n)$. Further, let $\delta>0$ be a sufficiently small constant to be specified later, and define the {\em near-perfect coupling event} from vertex $\Ver_2$ to be
	\eqn{
	\label{Cevent2-n-def-CM}
	\Ccal_n(2)=\big\{\big||\partial B_{k}^{\sss(G_n)}(\Ver_2)|-|\BP_k^{\sss(2)}|\big|\leq (\overline{m}_n^2/\ell_n)^{1+\delta}~\forall k\in[\bar{k}_n+1]\big\}.
	}
With $\underline{m}_n=o(\sqrt{n})$, we will later pick $\overline{m}_n$ such that $\underline{m}_n\overline{m}_n\gg n$ to reach our conclusion. The following lemma shows that also $\Ccal_n(2)$ occurs whp:

\begin{lemma}[Coupling beyond Lemma \ref{lem-coupling-CM}]
\label{lem-coupling-CM-beyond}
Consider $\CMnd$ and let $\overline{m}_n^2/\ell_n\rightarrow \infty$. Then, for every $\delta>0$, with the near-perfect coupling event $\Ccal_n(2)$ defined in \eqref{Cevent2-n-def-CM},
	\eqn{
	\prob(\Ccal_n(2))=1-o(1).
	}
\end{lemma}

For the proof, we refer to Appendix \ref{app-coupling-B-BP}.
We now define the {\em successful coupling event} $\Ccal_n$ to be
	\eqn{
	\label{Cevent-n-def-CM}
	\Ccal_n=\Ccal_n(1)\cap\Ccal_n(2),\qquad \text{so that} \qquad \prob(\Ccal_n)=1-o(1).
	}

\paragraph{\bf Branching process neighbourhood growth.}
The previous step relates the graph exploration process to two independent $n$-dependent unimodular branching processes $(\BP_k^{\sss(1)}, \BP_k^{\sss(2)})_{k\geq 1}$. In this step, we investigate the growth of these branching processes. Fix $r$, and denote $b_0^{\sss(i)}=|\BP_r^{\sss(i)}|$, which we assume to be at least $r$ (and which is true on the event $\Ccal_n\cap \{|\partial B_r^{\sss(G_n)}(\Ver_1)|, |\partial B_r^{\sss(G_n)}(\Ver_2)|\geq r\}$). 

Let $\nu_n=\frac{1}{\ell_n}\sum_{v\in [n]} d_v(d_v-1)$ denote the expected forward degree of a uniform half-edge in $\CMnd$, which also equals the expected offspring of the branching processes $(\BP_k^{\sss(i)})_{k\geq 0}$. Define
	\eqn{
	\label{Devent-n-def-CM}
	\Dcal_n=\big\{\underline{b}_k^{\sss(i)}\leq |\BP_{r+k}^{\sss(i)}| \leq \overline{b}_k^{\sss(i)}~\forall i\in [2], ~k\geq 0\big\},
	}
where the lower- and upper-bounding sequences $(\underline{b}_k^{\sss(i)})_{k\geq 0}$ and $(\overline{b}_k^{\sss(i)})_{k\geq 0}$ satisfy the recursions $\underline{b}_0^{\sss(i)}=\overline{b}_0^{\sss(i)}=b_0^{\sss(i)}$, while, for some $\alpha\in (\tfrac{1}{2},1)$,
	\eqn{
	\label{recursions-b}
	\underline{b}_{k+1}^{\sss(i)}=\nu_n\underline{b}_{k}^{\sss(i)}-(\overline{b}_{k}^{\sss(i)})^{\alpha},
	\qquad
	\overline{b}_{k+1}^{\sss(i)}=\nu_n\overline{b}_{k}^{\sss(i)}+(\overline{b}_{k}^{\sss(i)})^{\alpha}.
	}
The following lemma investigates the asymptotics of $(\underline{b}_k^{\sss(i)})_{k\geq 1}$ and $(\overline{b}_k^{\sss(i)})_{k\geq 1}$:

\begin{lemma}[Asymptotics of 	$\underline{b}_k^{\sss(i)}$ and $\overline{b}_k^{\sss(i)}$]
\label{lem-asy-bar-bns-CM}
Assume that $\lim_{n\rightarrow \infty} \nu_n=\nu>1$, and assume that $b_0^{\sss(i)}\geq r$. Then, there exists an $A=A_r>1$ with $A_r<\infty$ for $r$ large enough, such that, for all $k\geq 0$, the solutions to the recursion relations in \eqref{recursions-b} satisfy 
	\eqn{
	\overline{b}_k^{\sss(i)}\leq Ab_0^{\sss(i)}\nu_n^k,
	\qquad
	\underline{b}_k^{\sss(i)}\geq b_0^{\sss(i)}\nu_n^k/A.
	}
\end{lemma}
\proof First, obviously, $\overline{b}_k^{\sss(i)}\geq b_0^{\sss(i)}\nu_n^k$. Thus, since $\alpha<1$ and also using that $\nu_n>1$ for $n$ large enough,
	\eqan{
	\overline{b}_{k+1}^{\sss(i)}&=\nu_n\overline{b}_{k}^{\sss(i)}+(\overline{b}_{k}^{\sss(i)})^{\alpha}
	=\nu_n\overline{b}_{k}^{\sss(i)}(1+(\overline{b}_{k}^{\sss(i)})^{\alpha-1})\leq \nu_n\overline{b}_{k}^{\sss(i)}\big(1+r^{-(1-\alpha)}\nu_n^{-(1-\alpha)k}\big),
	}
By iteration, this implies the upper bound with $A$ replaced by $\bar{A}_r$ given by
	\eqn{
	\bar{A}_r=\prod_{k\geq 0} \big(1+r^{-(1-\alpha)}\nu_n^{-(1-\alpha)k}\big)<\infty.
	}
Further, note that $\lim_{r\rightarrow \infty} \limsup_{n\rightarrow \infty} \bar{A}_r=1$. For the lower bound, we use that $\overline{b}_k^{\sss(i)}\leq \bar{A}_rb_0^{\sss(i)}\nu_n^k$ to obtain
	\eqn{
	\underline{b}_{k+1}^{\sss(i)}\geq \nu_n\underline{b}_{k}^{\sss(i)}-\bar{A}_r^{\alpha}(b_0^{\sss(i)})^\alpha \nu_n^{\alpha k}.
	}
We now use induction to show that 
	\eqn{
	\underline{b}_{k}^{\sss(i)}\geq a_{k} b_0^{\sss(i)} \nu_n^k,
	}
where $a_0=1$ and 
	\eqn{
	\label{ak-rec-CM-prime}
	a_{k+1}=a_k -\bar{A}_r^{\alpha}r^{1-\alpha}\nu_n^{(\alpha-1) k-1}.
	}
The initialization follows, since $\underline{b}_{0}^{\sss(i)}=b_0^{\sss(i)}$ and $a_0=1$. To advance the induction hypothesis, we substitute the induction hypothesis to obtain that
	\eqan{
	\underline{b}_{k+1}^{\sss(i)}&\geq a_k b_0^{\sss(i)} \nu_n^{k+1}-\bar{A}_r^{\alpha}(b_0^{\sss(i)})^\alpha \nu_n^{\alpha k}\\
	&=b_0^{\sss(i)} \nu_n^{k+1} \big(a_k-\bar{A}_r^{\alpha}(b_0^{\sss(i)})^{\alpha-1} \nu_n^{(\alpha-1) k-1}\big)\nn\\
	&\geq b_0^{\sss(i)} \nu_n^{k+1} \big(a_k-\bar{A}_r^{\alpha}r^{\alpha-1} \nu_n^{(\alpha-1) k-1}\big)=a_{k+1}b_0^{\sss(i)} \nu_n^{k+1},\nn
	}
by \eqref{ak-rec-CM-prime}. Finally, $a_k$ is decreasing, and thus $a_k \searrow a\equiv 1/\underline{A}_r,$ where 
	\[
	\underline{A}_r^{-1}=1-\sum_{k\geq 0}\bar{A}_r^{\alpha}r^{-(1-\alpha)}\nu_n^{-(1-\alpha)k}<\infty
	\]
for $r$ large enough, so that the claim follows with $A=A_r=\max\{\bar{A}_r, \underline{A}_r\}$. Note further that $\lim_{r\rightarrow \infty} \limsup_{n\rightarrow \infty} \underline{A}_r=1$, since $\lim_{r\rightarrow \infty} \limsup_{n\rightarrow \infty} \bar{A}_r=1$.
\qed
\medskip

The following lemma shows that $\Dcal_n=\big\{\underline{b}_k^{\sss(i)}\leq |\BP_{r+k}^{\sss(i)}| \leq \overline{b}_k^{\sss(i)}~\forall i\in [2], ~k\geq 0\big\}$ occurs with probability close to 1 when first $n\rightarrow \infty$ followed by $r\rightarrow \infty$:

\begin{lemma}[$\Dcal_n$ occurs whp]
\label{lem-Devent-n-whp}
%Assume that $b_0^{\sss(i)}=|\BP_r^{\sss(i)}|\geq r$. Then
	\eqn{
	\lim_{r\rightarrow \infty}\liminf_{n\rightarrow \infty} \prob(\Dcal_n\mid |\BP_r^{\sss(i)}|\geq r)=1.
	}

\end{lemma}

\proof We will show that $\lim_{r\rightarrow \infty}\limsup_{n\rightarrow \infty}\prob(\Dcal_n^c)=0$. The inequalities $\underline{b}_0^{\sss(i)}\leq |\BP_{r}^{\sss(i)}| \leq \overline{b}_0^{\sss(i)}$ hold by definition. We thus write
	\eqn{
	\prob(\Dcal_n^c)\leq \sum_{k=1}^{\infty}\prob(\Dcal_{n,k}^c\cap\Dcal_{n,k-1}),
	}
where 
	\eqn{
	\Dcal_{n,k}=\big\{\underline{b}_k^{\sss(i)}\leq |\BP_{r+k}^{\sss(i)}|\leq \overline{b}_k^{\sss(i)}~\forall i\in [2]\big\}.
	}
Note that, when $|\BP_{r+k}^{\sss(i)}|>\overline{b}_k^{\sss(i)}$ and $|\BP_{r+k-1}^{\sss(i)}|\leq \overline{b}_{k-1}^{\sss(i)}$, by \eqref{recursions-b},
	\eqn{
	|\BP_{r+k}^{\sss(i)}|-\nu_n |\BP_{r+k-1}^{\sss(i)}|>\overline{b}_k^{\sss(i)}-\nu_n\overline{b}_{k-1}^{\sss(i)}=(\overline{b}_{k-1}^{\sss(i)})^{\alpha},
	}
while, when $|\BP_{r+k}^{\sss(i)}|<\underline{b}_k^{\sss(i)}$ and $|\BP_{r+k-1}^{\sss(i)}|\geq \underline{b}_{k-1}^{\sss(i)}$, again by \eqref{recursions-b},
	\eqn{
	|\BP_{r+k}^{\sss(i)}|-\nu_n|\BP_{r+k-1}^{\sss(i)}|<\underline{b}_k^{\sss(i)}-\nu_n\underline{b}_{k-1}^{\sss(i)}=(\overline{b}_{k-1}^{\sss(i)})^{\alpha},
	}
Thus, 
	\eqan{
	&\Dcal_{n,k}^c\cap\Dcal_{n,k-1}\\
	&\quad\subseteq\big\{\big||\BP_{r+k}^{\sss(1)}|-\nu_n|\BP_{r+k-1}^{\sss(1)}|\big|\geq (\overline{b}_{k-1}^{\sss(1)})^{\alpha}\big\}\cup
	\big\{\big||\BP_{r+k}^{\sss(2)}|-\nu_n|\BP_{r+k-1}^{\sss(2)}|\big|\geq (\overline{b}_{k-1}^{\sss(2)})^{\alpha}\big\}.\nn
	}
By the Chebychev inequality, conditionally on $\Dcal_{n,k-1}$,
	\eqan{
	&\prob\big(\big||\BP_{r+k}^{\sss(i)}|-\nu_n|\BP_{r+k-1}^{\sss(i)}|\big|\geq (\overline{b}_{k-1}^{\sss(i)})^{\alpha}\mid \Dcal_{n,k-1}\big)\\
	&\qquad =\expec\Big[\prob\Big(\big||\BP_{r+k}^{\sss(i)}|-\nu_n|\BP_{r+k-1}^{\sss(i)}|\big|\geq (\overline{b}_{k-1}^{\sss(i)})^{\alpha}~\Big|~ |\BP_{r+k-1}^{\sss(i)}|, \Dcal_{n,k-1}\Big)~\Big|~ \Dcal_{n,k-1}\Big]\nn\\
	&\qquad \leq \frac{\expec\Big[\Var(|\BP_{r+k}^{\sss(i)}|\mid |\BP_{r+k-1}^{\sss(i)}|, \Dcal_{n,k-1})\mid \Dcal_{n,k-1}\Big]}{(\overline{b}_{k-1}^{\sss(i)})^{2\alpha}}\leq \frac{\sigma_n^2\expec[|\BP_{r+k-1}^{\sss(i)}|\mid \Dcal_{n,k-1}]}{(\overline{b}_{k-1}^{\sss(i)})^{2\alpha}}
	\leq \sigma_n^2(\overline{b}_{k-1}^{\sss(i)})^{1-2\alpha},\nn
	}
where $\sigma_n^2$ is the variance of the offspring distribution given by
	\eqn{
	\sigma_n^2=\frac{1}{\ell_n}\sum_{v\in [n]}d_v(d_v-1)^2-\nu_n^2,
	}
which is uniformly bounded when all degrees are bounded by $b$. Thus, by the union bound for $i\in \{1,2\}$,
	\eqn{
	\prob(\Dcal_{n,k}^c\cap\Dcal_{n,k-1})\leq 2\sigma_n^2(\overline{b}_{k-1}^{\sss(i)})^{1-2\alpha},
	}
and we conclude that
	\eqn{
	\prob(\Dcal_n^c)\leq 2\sigma_n^2\sum_{k=1}^{\infty}\big((\overline{b}_{k-1}^{\sss(1)})^{1-2\alpha}+(\overline{b}_{k-1}^{\sss(2)})^{1-2\alpha}\big).
	}
The claim now follows from Lemma \ref{lem-asy-bar-bns-CM} and the fact that $\sigma_n^2\leq b(b-1)^2$ remains uniformly bounded.
\qed
	
%Let
%	\eqn{
%	\label{nun'-def}
%	\nu_n'=\frac{1}{\ell_n}\sum_{v\in [n']} d_v'(d_v'-1)=\frac{1}{\ell_n}\sum_{v\in [n]} d_v'(d_v'-1),
%	}
%where the last equality makes use of the fact that $d_v'=1$ for $v\in [n']\setminus[n]$, 

\paragraph{\bf Completion of the proof of Theorem \ref{thm-convtwolargestcomps} for bounded degrees.}
Recall \eqref{aim-PT-CM}. Also recall the definition of $\Ccal_n$ in \eqref{Cevent-n-def-CM}, \eqref{Cevent1-n-def-CM} and \eqref{Cevent2-n-def-CM}, and that of $\Dcal_n$ in \eqref{Devent-n-def-CM}. Let $\Gcal_n=\Ccal_n\cap \Dcal_n$ be the good event. By \eqref{Cevent-n-def-CM} and Lemma \ref{lem-Devent-n-whp},
	\eqn{
	\label{aim-PT-CM-a}
	\lim_{r\rightarrow \infty} \limsup_{n\rightarrow \infty} \prob(|\partial B_r^{\sss(G_n)}(\Ver_1)|, |\partial B_r^{\sss(G_n)}(\Ver_2)|\geq r, \Ver_1\nc \Ver_2; \Gcal_n^c)=0,
	}
so that it suffices to investigate $\prob(|\partial B_r^{\sss(G_n)}(\Ver_1)|, |\partial B_r^{\sss(G_n)}(\Ver_2)|\geq r, \Ver_1\nc \Ver_2; \Gcal_n)$. On $\Gcal_n$ (recall \eqref{Cevent2-n-def-CM}),
	\eqn{
	|\partial B_{\bar{k}_n}^{\sss(G_n)}(\Ver_2)|-|\BP_{\bar{k}_n}^{\sss(2)}|\geq -(\overline{m}_n^2/\ell_n)^{1+\delta}.
	}
Further, on $\Gcal_n$ (recall \eqref{Cevent1-n-def-CM}),
	\eqn{
	|\partial B_{\underline{k}_n}^{\sss(G_n)}(\Ver_1)|-|\BP_{\underline{k}_n}^{\sss(1)}|=0.
	}
On the event $\{\Ver_1\nc \Ver_2\}$, we must have that $\partial B_{\underline{k}_n}^{\sss(G_n)}(\Ver_1)\cap \partial B_{\bar{k}_n}^{\sss(G_n)}(\Ver_2)=\varnothing$. By Lemma \ref{lem-asy-bar-bns-CM}, on $\Gcal_n$ and when $\overline{m}_n^2/\ell_n\rightarrow \infty$ sufficiently slowly, $|\partial B_{\underline{k}_n}^{\sss(G_n)}(\Ver_1)|=\Theta(\underline{m}_n)$ and $|\partial B_{\bar{k}_n}^{\sss(G_n)}(\Ver_2)|=\Theta(\overline{m}_n)$.
The same bounds hold for the number of half-edges $Z_{\underline{k}_n}^{\sss(1)}$ and $Z_{\overline{k}_n}^{\sss(2)}$ incident to $\partial B_{\underline{k}_n}^{\sss(G_n)}(\Ver_1)$ and $\partial B_{\overline{k}_n}^{\sss(G_n)}(\Ver_2)$, respectively, since $Z_{\underline{k}_n}^{\sss(1)}\geq |\partial B_{\underline{k}_n+1}^{\sss(G_n)}(\Ver_1)|$ and $Z_{\overline{k}_n}^{\sss(2)}\geq |\partial B_{\overline{k}_n+1}^{\sss(G_n)}(\Ver_2)|$, so that, on $\Gcal_n$, also $Z_{\underline{k}_n}^{\sss(1)}=\Theta_{\sss \prob}(\underline{m}_n)$ and $Z_{\overline{k}_n}^{\sss(2)}=\Theta_{\sss \prob}(\overline{m}_n)$.

Conditionally on having paired some half-edges incident to $\partial B_{\underline{k}_n}^{\sss(G_n)}(\Ver_1)$, conditionally on none of them being paired to half-edges incident to $\partial B_{\overline{k}_n}^{\sss(G_n)}(\Ver_2)$, each further such half-edge has probability at least $1-Z_{\overline{k}_n}^{\sss(2)}/\ell_n$ to be paired to a half-edge incident to $\partial B_{\overline{k}_n}^{\sss(G_n)}(\Ver_2)$, thus creating a path between $\Ver_1$ and $\Ver_2$. The latter conditional probability is {\em independent} of the pairing of the earlier half-edges. Thus, the probability that $\partial B_{\underline{k}_n}^{\sss(G_n)}(\Ver_1)$ is not {\em directly} connected to $\partial B_{\overline{k}_n}^{\sss(G_n)}(\Ver_2)$ is at most
	\eqn{
	\Big(1-\frac{Z_{\underline{k}_n}^{\sss(1)}}{\ell_n}\Big)^{Z_{\overline{k}_n}^{\sss(2)}/2},
	}
since at least $Z_{\overline{k}_n}^{\sss(2)}/2$ pairings need to be performed. Since $Z_{\underline{k}_n}^{\sss(1)}=\Theta(\underline{m}_n)$ and $Z_{\overline{k}_n}^{\sss(2)}=\Theta(\underline{m}_n)$, this probability vanishes when $\underline{m}_n \overline{m}_n\gg n$. As a result, as $n\rightarrow \infty$,
	\eqn{
	\prob(|\partial B_r^{\sss(G_n)}(\Ver_1)|, |\partial B_r^{\sss(G_n)}(\Ver_2)|\geq r, \Ver_1\nc \Ver_2; \Gcal_n)=o(1),
	}
as required. This completes the proof of \eqref{almost-local-cond} for $\CMnd$ with bounded degrees, and even shows that $\gdist{\CMnd}(\Ver_1,\Ver_2)\leq 2r+\overline{k}_n+\underline{k}_n+1$ whp on the event that $|\partial B_r^{\sss(G_n)}(\Ver_1)|, |\partial B_r^{\sss(G_n)}(\Ver_2)|\geq r$, where $\gdist{\CMnd}(u,v)$ denotes the graph distance in $\CMnd$ between vertices $u,v\in[n]$.
%By the above perfect coupling to $(\BP_k^{\sss(1)}, \BP_k^{\sss(2)})_{k\leq k_n}$ to describe the sizes of the neighbourhood shells $(|\partial B_k^{\sss(G_n)}(\Ver_1|), |\partial B_k^{\sss(G_n)}(\Ver_2)|)_{k\leq k_n}$. For this, we denote $b_0^{\sss(i)}=|\BP_r^{\sss(i)}|$, which is at least $r$ on the event $\Ccal_n\cap \{|\partial B_r^{\sss(G_n)}(\Ver_1)|, |\partial B_r^{\sss(G_n)}(\Ver_2)|\geq r\}$. 
\qed

\paragraph{\bf Extension of the proof to unbounded degrees.}
To extend the proof to configuration models whose degrees satisfy Conditions \ref{cond-degrees-regcondII}(a)-(b), we apply a {\em degree-truncation technique} that allows us to go from such a configuration model to one whose degrees are uniformly bounded. This result makes the proof of the `giant is almost local' condition in \eqref{almost-local-cond-rep} simpler, and is interesting in its own right:

\begin{theorem}[Degree truncation for configuration models]
\label{thm-degree-truncation-CM}
Consider $\CMnd$. Fix $b\geq 1$. There exists a related configuration model $\CMndprime$ that is coupled to $\CMnd$ and satisfies that
\begin{itemize}

\item[(a)] the degrees in $\CMndprime$ are truncated versions of those in $\CMnd$, i.e., $d_v'=(d_v\wedge b)$ for $v\in[n]$, and $d_v'=1$ for $v\in[n']\setminus [n]$;

\item[(b)] the total degree in $\CMndprime$  is the same as that in $\CMnd$, i.e., $\sum_{v\in[n']}d_v'=\sum_{v\in [n]} d_v$;

\item[(c)] for all $u,v\in [n]$, if $u$ and $v$ are connected in $\CMndprime$, then so are $u$ and $v$ in $\CMnd$.
 
\end{itemize}
\end{theorem}

The proof of Theorem \ref{thm-degree-truncation-CM} is in Appendix \ref{app-degree-truncation}. 

\paragraph{Completion of the proof of Theorem \ref{thm-convtwolargestcomps} for unbounded degrees.} We use Theorem \ref{thm-degree-truncation-CM} with $b$ large. By Conditions \ref{cond-degrees-regcondII}(a)-(b), we can choose $b=b(\vep)$ so large that
	\eqn{
	\label{vep-def-truncation-degrees}
	\prob(D_n^\star >b)=\frac{1}{\ell_n} \sum_{v\in [n]} d_v \indic{d_v>b}\leq \frac{\vep n}{\ell_n}.
	}
This implies that $\CMndprime$ has at most $(1+\vep)n$ vertices, and the (at most $\vep n$) vertices in $[n']\setminus [n]$ all have degree 1, while the vertices in $[n]$ have degree $d_v'\leq b$. As a result, it suffices to prove Theorem \ref{thm-convtwolargestcomps}(a) for $\CMndprime$ instead, and, in the remainder of this proof, we assume that $\dmax'\leq b$ is uniformly bounded. We thus apply Theorem \ref{thm-giant-conv-LWC-structure} to $\CMndprime$. %This implies that $\CMndprime$ has at most $(1+\vep)n$ vertices, and the (at most $\vep n$) extra vertices compared to $\CMnd$ all have degree 1, while the vertices in $[n]$ have degree $d_v'\leq b$. 
We denote the parameters of $\CMndprime$ by $(p_k')_{k\geq 1}, \xi'$ and $\zeta'$, respectively, and note that, for $\vep$ as in \eqref{vep-def-truncation-degrees}, when $\vep\searrow 0$,
	\eqn{
	p_k'\rightarrow p_k,\qquad
	\xi'\rightarrow \xi,
	\qquad
	\zeta'\rightarrow \zeta,
	}
Further, with $\Cmax'$ the largest connected component in $\CMndprime$, by \eqref{vep-def-truncation-degrees},
	\eqn{
	|\Cmax'|\leq |\Cmax|+\vep n,
	}
so that
	\eqn{
	\prob(|\Cmax|\geq n(\zeta'-2\vep))\rightarrow 1,
	}
since $|\Cmax'|/n\convp \zeta'$. This proves the required lower bound on $|\Cmax|$, while the upper bound follows from Corollary \ref{cor-giant-UB}.\qed

\subsection{Small-world nature of the configuration model}

In this section, we use the above proof to study the {\em typical distances} in $\CMnd$. Such distances have attracted considerable attention, see e.g., \cite{DerMonMor12, DerMonMor17, HofHooVan05a, HofHooZna04a, HofKom17}. Here, we partially reprove a result from \cite{HofHooVan05a}:
\label{sec-small-world}
\begin{theorem}[Typical distances in $\CMnd$ for finite-variance degrees]
\label{thm-averdistCM>3}
Consider the configuration model $\CMnd$ subject to Conditions \ref{cond-degrees-regcondII}(a)-(b) and where $\expec[D_n^2]\rightarrow \expec[D^2]<\infty$ with
	$
	\nu\in (1,\infty).
	$
Then, conditionally on $\gdist{\CMnd}(\Ver_1,\Ver_2)<\infty$,
    \eqn{
    \label{Hn-conv-CM>3}
    \gdist{\CMnd}(\Ver_1,\Ver_2)/\log{n}\convp 1/\log{\nu}.
    }
\end{theorem}

%Theorem \ref{thm-averdistCM>3} shows that the typical distances in $\CMnd$ are of order $\log_\nu{n}$, and is thus similar in spirit as Theorem \ref{thm-averdistNR>3}. We shall see that also its proof is quite similar.

\proof Recall that we condition on $\gdist{\CMnd}(\Ver_1,\Ver_2)<\infty$. Since $|\cluster_{\sss(2)}|/n\convp 0$, the majority of pairs of vertices $u,v$ satisfying $\gdist{\CMnd}(\Ver_1,\Ver_2)<\infty$ satisfy that $u,v\in \Cmax$. Thus, from now on, we will assume that $\Ver_1,\Ver_2\in \Cmax$, and, particularly, their component sizes are large.

\paragraph{Proof of upper bound in Theorem \ref{thm-averdistCM>3} for bounded degrees.} We again start by proving the result for bounded degrees. The proof in \eqref{aim-PT-CM-a} shows that for pairs of vertices with large connected components, whp as first $n\rightarrow \infty$ followed by $r\rightarrow \infty$,
	\eqn{
	\gdist{\CMnd}(\Ver_1,\Ver_2)\leq \underline{k}_n+\bar{k}_n+1=\frac{\log{n}}{\log{\nu_n}}(1+\op(1)).
	} 
Indeed, Lemma \ref{lem-asy-bar-bns-CM} implies that $\underline{k}_n=(1+\op(1))\log{\underline{m}_n}/\log{\nu_n}$ and $\bar{k}_n=(1+\op(1))\log{\overline{m}_n}/{\log{\nu_n}}$ on the good event $\Gcal_n=\Ccal_n\cap \Dcal_n$, so that 
	\eqn{
	\underline{k}_n+\bar{k}_n+1=(1+\op(1))\log{(\underline{m}_n\overline{m}_n)}/{\log{\nu_n}}=(1+\op(1))\log{n}/{\log{\nu_n}}
	}
when $\underline{m}_n\overline{m}_n/n\rightarrow \infty$ slowly enough. Further, $\nu_n\rightarrow \nu$, so that, for every $\vep>0$,
	\eqn{
	\prob(\gdist{\CMnd}(\Ver_1,\Ver_2)/\log{n}\leq (1+\vep)/\log{\nu}\mid \gdist{\CMnd}(\Ver_1,\Ver_2)<\infty)=1-o(1).
	}
This proves the upper bound on $\gdist{\CMnd}(\Ver_1,\Ver_2)$ for uniformly bounded degrees.

\paragraph{Proof of of upper bound in Theorem \ref{thm-averdistCM>3} for unbounded degrees.} We next extend this argument to degrees under Conditions \ref{cond-degrees-regcondII}(a)-(b) and when $\expec[D_n^2]\rightarrow \expec[D^2]<\infty.$ We again rely on the degree-truncation argument in Theorem \ref{thm-degree-truncation-CM}, and let $\CMndprime$ be the graph obtained in Theorem \ref{thm-degree-truncation-CM}. Recall from \eqref{vep-def-truncation-degrees} that $\CMndprime$ has at most $(1+\vep)n$ vertices, and the (at most $\vep n$) vertices in $[n']\setminus [n]$ all have degree 1, while the vertices in $[n]$ have degree $d_v'\leq b$. 
\smallskip

Let $\Cmax'$ be the giant in $\CMndprime$, and $\Cmax$ that in $\CMnd$. By Theorem \ref{thm-degree-truncation-CM}, pairs $u,v$ of vertices in $[n]$ that are connected in $\CMndprime$ are also connected in $\CMnd$. Thus, all vertices in $\Cmax'\cap [n]$ are also connected in $\CMnd$, and therefore are also part of $\Cmax$. As a result, differences in $|\Cmax'|$ and $|\Cmax|$ only arise through the addition (at most $\vep n$) degree-one vertices in $[n']\setminus [n]$. We conclude that the giants $\Cmax'$ in $\CMndprime$ and $\Cmax$ in $\CMnd$ satisfy that $|\Cmax'\Delta \Cmax|\leq \vep n$ whp for $b$ sufficiently large, where $\Delta$ denotes the symmetric difference between sets. Thus, whp, a pair of vertices in $\Cmax$ is also in $\Cmax'$. 
\smallskip

Finally, let $\nu_n'=\sum_{v\in [n']}d_v'(d_v'-1)/\sum_{v\in [n']}d_v'$ denote the expected forward degree of a half-edge for $\CMndprime$. When Conditions \ref{cond-degrees-regcondII}(a)-(b) hold and $\expec[D_n^2]\rightarrow \expec[D^2]<\infty$, also $\nu_n'\rightarrow \nu'=\nu'(\vep)$, where $\nu'(\vep)\rightarrow \nu$ when $\vep\searrow 0$. We conclude that the upper bound on $\gdist{\CMnd}(\Ver_1,\Ver_2)$ for unbounded degrees under Conditions \ref{cond-degrees-regcondII}(a)-(b) follows from that for uniformly bounded degrees.

\paragraph{Proof of of lower bound in Theorem \ref{thm-averdistCM>3}} The lower bound on $\gdist{\CMnd}(\Ver_1,\Ver_2)$ applies more generally, and follows from the fact that
	\eqn{
	\prob(\gdist{\CMnd}(\Ver_1,\Ver_2)\leq k)=\expec[|B^{\sss(G_n)}_k(\vertex_1)|/n],
	}
together with the fact that, for $k\geq 1$,
	\eqn{
	\expec[|\partial B^{\sss(G_n)}_k(\vertex_1)|]\leq \frac{\ell_n}{n}\nu_n^{k-1}
	}
by \cite[Lemma 5.1]{Jans09b}. Thus, with $k_n=\lfloor (1-\vep)\log{n}/{\log{\nu_n}}\rfloor$,
	\eqn{
	\prob(\gdist{\CMnd}(\Ver_1,\Ver_2)\leq (1-\vep)\log{n}/{\log{\nu_n}})\leq \frac{1}{n}+\sum_{k=1}^{k_n} \frac{\ell_n}{n}\nu_n^{k-1}
	\leq  \frac{1}{n} +\frac{\ell_n}{n} \frac{\nu_n^{k_n}-1}{\nu_n-1}=o(1),
	}
as required.
\qed

\section{Discussion and open problems}
\label{sec-disc-open-problems}
In this section, we discuss the history of our paper, our proof and results, and state some open problems.

\paragraph{History of this paper.} This paper first appeared on the ar{X}iv on March 23, 2021, and a revised version on June 12, 2023. Since then, it has been incorporated into the book \cite{Hofs24}, appearing in 2024; see \cite[Section 2.6]{Hofs24}, where the main results on the `almost local' nature of the giant have been recorded, and \cite[Chapter 4]{Hofs24}, where its consequences for the configuration model were described. 

The way the research community perceives new proofs in books is not consistent. Some researchers rightfully argue that books are not reviewed in the same way as papers are, while others appreciate new and elegant proofs in books. Since this paper is now drawing ample citations (e.g., 20 on Google Scholar on September 29, 2025), and a follow-up paper extending some results to directed random graphs have already appeared \cite{HofPan25}, I personally believe that it serves the community to publish this paper in the refereed literature.  

We have also made changes to this paper compared to the versions in \cite{Hofs24}. Indeed, the proof of necessity of the `giant-is-almost-local' condition is given in Remark \ref{rem-nec-suf}. Further, the current proof of Lemma \ref{lem-almost-local-cond-rep} is slightly simplified, in that only 1 condition remains in \eqref{error-terms-simplification}. Several typos have been corrected. Corollary \ref{cor-LLN-outside} has been added (replacing an erroneous version  based on bounds on the second largest component in the 2023 version), and its interesting consequences based on \cite{JorKomMit25} are explained in greater detail. Lemma \ref{lem-coupling-CM} has been made more precise. Finally, we corrected some typos in \cite[Theorem 9.38]{Hofs24}; see Theorem \ref{thm-giant-GIRG}.

\paragraph{\bf Minimal conditions for the proof.} Inspection of the proof of Corollary \ref{cor-giant-UB} and Theorem \ref{thm-giant-conv-LWC} shows that $Z_{\sss \geq k}/n\convp \zeta_{\sss\geq k}$, with $\zeta_{\sss\geq k}\searrow \zeta$, combined with \eqref{almost-local-cond}, suffices to obtain Theorem \ref{thm-giant-conv-LWC}. Theorem \ref{thm-giant-conv-LWC-structure} requires a little more (see \eqref{ZAgeqk-conv}), while only Theorem \ref{thm-giant-LWL} requires the full local convergence in probability. The use of $Z_{\sss \geq k}$ to study connected components (particularly close to criticality) has a long history, and is implicitly present in \cite{BorChaKesSpe99}, while being formally introduced in the high-dimensional percolation context in \cite{BorChaHofSlaSpe05a}, see \cite{BorChaHofSlaSpe06, HofNac13, HofNac17} for applications to percolation on the $n$-cube. It was used in \cite{Hofs09a} to study the critical behaviour for rank-1 inhomogeneous random graphs, and \cite{Hofs17} for the phase transition on the \erdos.

\paragraph{\bf Related models.} We use the configuration model as a proof of concept for the method in this paper, and gave several further examples using Corollary \ref{cor-LLN-outside} in Section \ref{sec-second-largest}. This shows that our approach is quite versatile, and here we give some further examples in the literature where the method is used. In \cite[Section 2.5]{Hofs24}, the method is performed for the \erdos. That proof can easily be adapted to finite-type inhomogeneous random graphs as studied in the seminal paper by Bollob\'as, Janson and Riordan \cite{BolJanRio07}, as performed in \cite[Chapters 3 and 6]{Hofs24}. In turn, such an approach also yields the type distribution of the giant, as identified in greater generality in \cite[Lemma 4.10]{JanRio12}.

\paragraph{\bf Relations to the literature.} The weak bound in Corollary \ref{cor-giant-UB} is folklore, and has played a crucial role in many proofs. Often, this part is combined with a {\em sprinkling} or {\em two-round exposure} argument to prove a matching lower bound. In such an argument, one generally proves a lower bound on the number of vertices in large clusters when ignoring a small proportion of the edges. After `sprinkling' these additional edges, one then proves that most of these largish connected components merge to create a giant component. Our proof is somewhat similar, but instead relies on the necessary and sufficient condition in \eqref{almost-local-cond}. It would be interesting to investigate whether \eqref{almost-local-cond} can also be {\em proved} using sprinkling arguments.

\paragraph{\bf Local convergence.} There are many models for which local convergence has been established, and thus our results might be applied. For inhomogeneous random graphs, local convergence was not established explicitly in the seminal paper by Bollob\'as, Janson and Riordan \cite{BolJanRio07}, but the methodology is very related and does allow for a simple proof of it (see \cite[Chapter 3]{Hofs24} for a formal proof, and \cite{DemMon10a} for a proof for \erdos). For the configuration model, and the related uniform random graph with prescribed degrees, it was proved explicitly by Dembo and Montanari \cite{DemMon10a}, see also \cite{Bord16, DemMon10b} as well as \cite[Chapter 4]{Hofs24} for related proofs. Berger, Borgs, Chayes and Saberi \cite{BerBorChaSab14} establish local convergence for preferential attachment models (see also \cite[Chapter 5]{Hofs24} and \cite{GarHazHofRay22} for some extensions). A related preferential attachment model with conditionally independent edges is treated by Dereich and M\"orters \cite{DerMor13}, who again do not formally prove local convergence, but state a highly related result. Random intersection graphs are treated in \cite{Kura15}, see also \cite{HofKomVad18}. In most of these models, the size of the giant is established, though mostly in ways that are quite different from the current proof (see \cite{Hofs24} for an extensive overview). It would be of interest to see whether the current proof simplifies some of these analyses.

\paragraph{\bf Limiting properties that do, or do not, follow from local convergence.}
As stated in the introduction, local convergence is a versatile tool. The number of connected components, the clustering coefficient, local neighbourhoods, and edge-neighbourhoods are all local quantities (see \cite[Chapter 2]{Hofs24} for many examples of local quantities). In some cases, some extra local conditions (typically on the degree distribution) need to be made in order to be able to use such results. 

Some less obvious properties also converge when the graph converges locally. An important and early example is the Ising model partition function on locally tree-like graphs  \cite{DemMon10a}. Also the PageRank distribution is local \cite{GarHofLit20}. A property that is almost local is the density of the densest subgraph in a random graph, as shown by Anantharam and Salez \cite{AnaSal16}. Lyons \cite{Lyon05} shows that the exponential growth rate of the number of spanning trees of a finite connected graph can be computed through the local limit. See also \cite{Sale13b} for weighted spanning subgraphs, \cite{GamNowSwi06} for maximum-weight independent sets, and \cite{BhaEvaSen12} for the limiting spectral distribution of the graph adjacency matrix.

Our approach shows that while the proportion of vertices in the giant is not a local property, it is `almost local' in the sense that one additional necessary and sufficient condition does identify the limiting proportion of vertices in the giant in terms of the local limit. 
It would be interesting to investigate whether the idea of `almost local' properties can be extended to other properties as well. For example, while Theorem \ref{thm-averdistCM>3} suggests that also logarithmic distances are an almost local property, the statement in Theorem \ref{thm-averdistCM>3} follows from the proof of Theorem \ref{thm-convtwolargestcomps}, rather than from a general `almost local' result as in Theorem \ref{thm-giant-conv-LWC}. It would be interesting to explore this further, to see whether also related typical distance results might be proved in a similar way (see \cite{DerMonMor12, DerMonMor17, HofHooVan05a, HofHooZna04a, HofKom17} for such related results).

\paragraph{\bf Percolation on finite graphs.} Another area where our results may be useful is percolation on finite graphs. There, it is often not even clear how to precisely {\em define} the critical value or critical behaviour. We refer to Janson and Warnke \cite{JanWar18} and Nachmias and Peres \cite{NacPer07b} for extensive discussions on the topic. Percolation on random graphs has also attracted substantial attention, see for example Janson \cite{Jans09c} or Fountoulakis \cite{Foun07} for the derivation of the limiting percolation threshold for the configuration model. This is related to the locality of the percolation threshold as investigated by Benjamini, Nachmias and Peres \cite{BenNacPer11} in the context of $d$-regular graphs that have large girth, and inspired by a question by Oded Schramm in 2008. Indeed, Schramm conjectured that the critical percolation threshold on a converging sequence of infinite graphs $G_n$ should converge to that of the graph limit. In a transitive setting, local convergence can be used also on infinite graphs (as every vertex is basically the same, thus skipping the necessity of drawing a root uar). See also \cite{AloBenSta04, BenBouLugRos12, Nach09} for related work on sharp threshold for the existence of a giant in the context of expanders. 

Recently, Alimohammadi, Borgs and Saberi \cite{AliBorSab23} brought the discussion significantly forward, by showing that the percolation critical value is indeed local for {\em expanders}. Further, interestingly, they applied their results to the Barab\'asi-Albert preferential attachment model \cite{BarAlb99}, in the version of Bollob\'as et al.\ in \cite{BolRioSpeTus01}, to show that $p_c=0$ and identify the giant for all $p>0$. The proof relies on a sprinkling argument. It would be interesting to study the relation between our central assumption in \eqref{almost-local-cond} and $(G_n)_{n\geq 1}$ being a sequence of expander graphs. Further, it would be very interesting to see whether local convergence has consequences for the {\em critical} behaviour of percolation on locally convergent graph sequences.

%\paragraph{\bf From typical vertices to rare vertices}

%%%%%%%%%%%%%%%%%%%%%%%%%%%%%%%%%%%%%%%%%%%%%%
\begin{appendix}
 
\section{Further ingredients for the configuration model}
\label{app-A}
In this section, we provide proofs of technical ingredients used in the proof of Theorem \ref{thm-convtwolargestcomps}. This section is organised as follows. In Appendix \ref{app-LC-CM}, we give some more details on local convergence in probability for the configuration model. In Appendix \ref{app-degree-truncation}, we prove Theorem \ref{thm-degree-truncation-CM} that is used to show that it suffices to prove Theorem \ref{thm-convtwolargestcomps} for {\em uniformly bounded} degrees. In Appendix \ref{app-coupling-B-BP}, we investigate the coupling of the graph exploration process for $\CMnd$ and the $n$-dependent unimodular branching process. 

\subsection{Local convergence of the configuration model}
\label{app-LC-CM}
As mentioned already below Theorem \ref{thm-convtwolargestcomps}, there are several ways in which local convergence in probability, as formalised in \eqref{LCP-def}, can be proved for the configuration model. Such proofs often consist of two steps. In the first, we investigate the expected number of vertices whose $r$-neighbourhood is isomorphic to a specific tree. In the second, we prove that this number is highly concentrated. Convergence of the mean follows from the coupling result in Lemma \ref{lem-coupling-CM} below (see the discussion right below it). Concentration can be proved using martingale techniques (by pairing edges one by one, and using Azuma-Hoefding concentration bounds), see \cite{DemMon10a} or \cite{Bord16}. Alternatively, one can use a second-moment method and results similar to Lemma \ref{lem-coupling-CM}. See \cite[Section 4.1]{Hofs24} for such an approach. We omit further details.

\subsection{A useful degree-truncation argument for the configuration model}
\label{app-degree-truncation}

%\begin{remark}[Truncation of degrees in range]
%\label{rem-degree-truncation-range-CM}
%{\rm The construction that proves Theorem \ref{thm-degree-truncation-CM} is highly flexible, and also allows for a degree truncation that maintains restrictions on the minimal degree $\dmin$.
%Indeed, fix $b\geq 1$. There exists a related configuration model $\CMndprime$ satisfying (b) and (c) in Theorem \ref{thm-degree-truncation-CM}, while (a) is replaced with $d_v'=d_v$ when $d_v<2b$ and $d_v'=b$ when $d_v\geq 2b$ for $v\in[n]$, and $b\leq d_v'<2b$ for $v\in[n']\setminus [n]$, so that $\dmin'=\min_{v\in[n']} d_v'\geq \dmin\wedge b$.\hfill $\blacksquare$
%}
%\end{remark}
In this section, we prove Theorem \ref{thm-degree-truncation-CM}. The proof relies on an `explosion' or `fragmentation' of the vertices $[n]$ in $\CMnd$ inspired by \cite{Jans09c}. Label the half-edges from 1 to $\ell_n$ in an arbitrary way. We go through the vertices $v\in[n]$ one by one. When $d_v\leq b$, we do nothing. When $d_v>b$, then we let $d_v'=b$, and we keep the $b$ half-edges with the lowest labels. The remaining $d_v-b$ half-edges are exploded from vertex $v$, in that they are incident to vertices of degree 1 in $\CMndprime$, and are given vertex labels above $n$. We give the exploded half-edges the remaining labels of the half-edges incident to $v$. Thus, the half-edges receive labels both in $\CMnd$ as well as in $\CMndprime$, and the labels of the half-edges incident to $v\in[n]$ in $\CMndprime$ are a subset of those in $\CMnd$. In total, we thus create an extra $n^+=\sum_{v\in [n]} (d_v-b)_+$ `exploded' vertices of degree 1, and $n'=n+n^+$.  
\smallskip

We then pair the half-edges randomly, in the same way in $\CMnd$ as in $\CMndprime$. This means that when the half-edge with label $x$ is paired to the half-edge with label $y$ in $\CMnd$, the half-edge with label $x$ is also paired to the half-edge with label $y$ in $\CMndprime$, for all $x,y\in[\ell_n]$.

We now check parts (a)-(c). Obviously parts (a) and (b) follow. For part (c), we note that all created vertices have degree 1. Further, for vertices $u,v\in[n]$, if there exists a path in $\CMndprime$ connecting them, then the intermediate vertices have degree at least 2, so that they cannot correspond to exploded vertices. Thus, the same path of paired half-edges also exists in $\CMnd$, so that $u$ and $v$ are connected in $\CMnd$ as well.
%
%We conclude by adapting the construction to prove the statement in Remark \ref{rem-degree-truncation-range-CM}. We again go through the vertices $v\in[n]$ one by one. When $d_v\leq 2b$, we do nothing. When $d_v>2b$, then we let $d_v'=b$, and we keep the $b$ half-edges with the lowest labels. The remaining $d_v-b$ half-edges are exploded from vertex $v$, in that they are incident to `exploded' vertices that all have degree $b$ in $\CMndprime$, possibly except for one vertex that has degree in $[b,2b)$, and are given vertex labels above $n$. We again give the exploded half-edges the remaining labels of the half-edges incident to $v$. This identifies the desired construction for Remark \ref{rem-degree-truncation-range-CM}.
\qed

\subsection{Coupling of configuration neighbourhoods and branching processes}
\label{app-coupling-B-BP}
In this section, we prove Lemma \ref{lem-coupling-CM} and Remark \ref{rem-extensions-lemma-coupling-CM}, as well as Lemma \ref{lem-coupling-CM-beyond}.

\begin{proof}[Proof of Lemma \ref{lem-coupling-CM} and Remark \ref{rem-extensions-lemma-coupling-CM}] 
Fix $m$. We explain how one can {\em jointly} construct $(\hatGraphn(t),\hatTreen(t))_{t\in[m]}$ {\em given} that we have already constructed $(\hatGraphn(t),\hatTreen(t))_{t\in[m-1]}$. We again label the half-edges from 1 to $\ell_n$ in an arbitrary way, and use these labels in our construction. At any moment in time, the set of half-edges is divided into paired half-edges and unpaired half-edges. Further, for each half-edge, we record its status as {\em real} or {\em ghost}, where the real half-edges correspond to the ones present in {\em both} $\hatGraphn(m-1)$ and $\hatTreen(m-1)$, while the ghost half-edges are only present in $\hatGraphn(m-1)$ or $\hatTreen(m-1)$. We note that, in $\hatTreen$, labelled half-edges with labels in $[\ell_n]$ can be used multiple times, and each of these occurrences of the half-edge receives a label ghost or real. By construction, at most one of the (possibly multiple) copies of a half-edge in $\hatTreen$ can have the status real, all others are ghosts. In $\hatGraphn$, each half-edge can only occur at most once, and the status of the half-edge can then be real or ghost (depending on whether it also arises in $\Treen$ or not). Each real half-edge in $\Graphn$ will have a real partner in $\Treen$, while the ghost half-edges do not have partners in the other exploration. We next explain how to extend the joint construction by one vertex to go from $(\hatGraphn(t),\hatTreen(t))_{t\in[m-1]}$ to $(\hatGraphn(t),\hatTreen(t))_{t\in[m]}$.

To obtain $\hatGraphn(m)$, we take the unpaired half-edge $x_m$ in $\hatGraphn(m-1)$ with smallest label. When this half-edge has the ghost status, then we draw a uniform {\em unpaired} half-edge $y_m'$ and pair $x_m$ to $y_m'$ to obtain $\hatGraphn(m)$, and we give $y_m'$ and all sibling half-edges of $y_m'$ the ghost status (where we recall that the sibling half-edges of a half-edge $y$ are those half-edges unequal to $y$ that are incident to the same vertex as $y$ is). 

When the half-edge $x_m$ has the real status, it needs to be paired both in $\hatGraphn(m)$ and $\hatTreen(m)$. To obtain $\hatGraphn(m)$, this half-edge needs to be paired to a uniform unpaired half-edge unequal to $x_m$, i.e., one that has not been paired so far. For $\hatTreen(m)$, this restriction does not hold. We now show how these two choices can be conveniently coupled.

For $\hatTreen(m)$, we draw a uniform half-edge $y_m$ from the collection of {\em all} half-edges, independently of the past. Let $U_m$ denote the vertex to which $y_m$ is incident. We then let the $m$th individual in $(\hatTreen(t))_{t\in[m-1]}$ have precisely $d_{U_m}-1$ children. Note that $d_{U_m}-1$ has the same distribution as $D_n^\star-1$ and, by construction, the collection $\big(d_{U_t}-1\big)_{t\geq 1}$ is iid. This constructs $\hatTreen(m)$, except for the statuses of the sibling half-edges incident to $U_t$, which we describe below.

For $\hatGraphn(m)$, when $y_m$ is unpaired, i.e., has not yet been paired in $(\hatGraphn(t))_{t\in[m-1]}$, and $y_m$ is unequal to $x_m$, then we let $x_m$ be paired to $y_m$ in $\hatGraphn(m)$, and we have thus also constructed $(\hatGraphn(t),\hatTreen(t))_{t\in[m]}$ with the correct law. We give $y_m$ and all the other half-edges of $U_m$ the status real when $U_m$ has not yet appeared in $\hatGraphn(m-1)$, and otherwise we give them the ghost status. The latter case implies that a cycle appears in $(\hatGraphn(t))_{t\in[m]}$. By construction, such a cycle does not occur in $(\hatTreen(t))_{t\in[m]}$, where re-used vertices are simply repeated several times, and as are their half-edges.

A difference in the coupling arises when $y_m$ has already been paired in $(\hatGraphn(t))_{t\in[m-1]}$, in which case we give all the sibling half-edges of $U_t$ the ghost status. For $\hatGraphn(m)$, we draw a uniform {\em unpaired} half-edge $y_m'$ unequal to $x_m$, and pair $x_m$ to $y_m'$ instead to obtain $\hatGraphn(m)$. We give $y_m'$ and all the sibling half-edges of $y_m'$ the ghost status. Clearly, this gives rise to a difference between $\hatGraphn(m)$ and $\hatTreen(m)$. 

We continue the above exploration algorithm until it terminates at some time $T_n$. Since each step pairs exactly two half-edges, we have that $T_n=|E(\cluster(\Ver))|$, so that $T_n\leq \ell_n/2$. The final result is then $(\hatGraphn(t), \hatTreen(t))_{t\in[T_n]}$. At this moment, however, the branching process tree $(\hatTreen(t))_{t\geq 1}$ has {\em not} been fully explored, since all the tree vertices corresponding to ghost half-edges in $(\hatTreen(t))_{t\geq 1}$ have not been explored. We complete the tree exploration $(\hatTreen(t))_{t\geq 1}$ by iid drawing children of all the ghost tree vertices until the full tree is obtained.

We emphasise that the law of $(\hatTreen(t))_{t\geq 1}$ obtained above is not the same as that of $(\Treen(t))_{t\geq 1}$. Indeed, $(\Treen(t))_{t\geq 1}$ is obtained by a breadth-first exploration of the branching process tree, while, for $(\hatTreen(t))_{t\geq 1}$, the order in which half-edges are paired is chosen so that  $(\hatGraphn(t))_{t\in[T_n]}$ has the same law as the breadth-first graph exploration process $(\Graphn(t))_{t\in[T_n]}$. However, with $\sigma_n$ the first moment that a ghost half-edge appears in the construction, we have that $(\hatTreen(t))_{t\in[\sigma_n]}$ {\em does} have the same law as $(\Treen(t))_{t\in[\sigma_n]}$.
\medskip

There are two sources of differences between $(\hatGraphn(t))_{t\geq 0}$ and $(\hatTreen(t))_{t\geq 0}$:

\begin{description}
\item[{\bf Half-edge re-use:}] 
In the above coupling, a half-edge re-use occurs when $y_m$ had already been paired and is being re-used in the branching process. As a result, for $(\hatGraphn(t))_{t\in[m]}$, we need to redraw $y_m$ to obtain $y_m'$ that is instead used in $(\hatGraphn(t))_{t\in[m]}$;

\item[{\bf Vertex re-use:}] 
A vertex re-use occurs when $U_m=U_{m'}$ for some $m'<m$. In the above coupling, this means that $y_m$ is a half-edge that has not yet been paired in $(\hatGraphn(t))_{t\in[m-1]}$, but it is {\em incident} to a half-edge that has already been paired in $(\hatGraphn(t))_{t\in[m-1]}$.  In particular, the vertex $U_m$ to which it is incident has already appeared in $(\hatGraphn(t))_{t\in[m-1]}$, and it is being re-used in the branching process. In this case, a {\em copy} of $U_{m}$ appears in $(\hatTreen(t))_{t\in[m]}$, while a {\em cycle} appears in $(\hatGraphn(t))_{t\in[m]}$.
\end{description}

We continue by providing a bound on both contributions:

\paragraph{\bf Bounding half-edge re-uses.} 
Up to time $m-1$, exactly $2m-1$ half-edges are forbidden to be used by $(\hatGraphn(t))_{t\leq m}$. The probability that the half-edge $y_m$ equals one of these these $2m-1$ previously chosen half-edges is at most
	\eqn{
	\frac{2m-1}{\ell_n}.
	}
Hence the expected number of half-edge re-uses before time $m_n$ is at most
	\eqn{
	\label{half-edge-re-uses-CM}
	\sum_{m=1}^{m_n}\frac{2m-1}{\ell_n}=\frac{m_n^2}{\ell_n}=o(1),
	}
when $m_n=o(\sqrt{n}).$ Thus, by the first-moment method, the probability that a half-edge re-use occurs is also bounded by $m_n^2/\ell_n$.

\paragraph{\bf Bounding vertex re-uses.} The probability that vertex $i$ is chosen in the $m$th draw is equal to $d_i/\ell_n$. The probability that vertex $i$ is drawn twice before time $m_n$ is at most
	\eqn{
	\frac{m_n(m_n-1)}{2}\frac{d_i^2}{\ell_n^2}.
	}
By the union bound, the probability that a vertex re-use has occurred before time $m_n$ is at most
	\eqn{
	\label{vertex-re-uses-CM}
	\frac{m_n(m_n-1)}{2\ell_n}\sum_{i\in[n]} \frac{d_i^2}{\ell_n}= m_n^2 \frac{n}{\ell_n^2}\expec[D_n^2]=o(1),
	}
by Conditions~\ref{cond-degrees-regcondII}(a)-(b), since $\ell_n=\Theta(n)$, and when $m_n=o(\sqrt{n/\expec[D_n^2]})$. This completes the coupling part of Lemma \ref{lem-coupling-CM}, including the bound on $m_n$ as formulated in Remark \ref{rem-extensions-lemma-coupling-CM}. It is straightforward to check that the exploration can be performed from the two sources $(\Ver_1,\Ver_2)$ independently, thus establishing the requested coupling to two independent branching processes claimed in Remark \ref{rem-extensions-lemma-coupling-CM}.
\end{proof}
%\medskip

Finally, we adapt the above argument to prove Lemma \ref{lem-coupling-CM-beyond}:
\begin{proof}[Proof of Lemma \ref{lem-coupling-CM-beyond}] Define $a_n=(\overline{m}_n^2/\ell_n)^{1+\delta}$ where $\delta>0$. 
Recall from \eqref{Cevent2-n-def-CM} that 
		\eqan{
		\Ccal_n(2)\equiv\Ccal_n(2,1)\cap \Ccal_n(2,2)&=\big\{(|\partial B_k^{\sss(G_n)}(\Ver_2)|)_{k\leq \underline{k}_n}=(|\BP_k^{\sss(2)}|)_{k\leq \underline{k}_n}\big\}\\
		&\qquad\cap \big\{\big||\partial B_{k}^{\sss(G_n)}(\Ver_2)|-|\BP_k^{\sss(2)}|\big|\leq (\overline{m}_n^2/\ell_n)^{1+\delta}~\forall k\in[\overline{k}_n+1]\big\},\nn		}
where $\Ccal_n(2,1)$ occurs whp. Also recall the definition of $\Dcal_n$ right above Lemma \ref{lem-Devent-n-whp}.

We apply a first-moment method and bound
	\eqan{
	\label{prob-Cevent2-compl}
	&\prob(\Ccal_n(2,2)^c\cap \Ccal_n(2,1)\cap \Dcal_n)\\
	&\quad\leq \frac{1}{a_n}\expec\Big[\indicwo{\Ccal_n(2,1)}\big||B_{\bar{k}_n}^{\sss(G_n)}(\Ver_2)|-|\BP_{\bar{k}_n}^{\sss(2)}|\big|\Big]\nn\\
	&\quad=\frac{1}{a_n}\expec\Big[\indicwo{\Ccal_n(2,1)\cap \Dcal_n}\big(|B_{\bar{k}_n}^{\sss(G_n)}(\Ver_2)|-|\BP_{\bar{k}_n}^{\sss(2)}|\big)_++\big(|B_{\bar{k}_n}^{\sss(G_n)}(\Ver_2)|-|\BP_{\bar{k}_n}^{\sss(2)}|\big)_-\Big]\nn\\
	&\quad \leq \frac{1}{a_n}\expec\Big[\indicwo{\Ccal_n(2,1)\cap \Dcal_n}\big(|\hatGraphn(m_n)|-|\hatTreen(m_n)|\big)_++\indicwo{\Ccal_n(2,1)\cap \Dcal_n}\big(|\hatGraphn(m_n)|-|\hatTreen(m_n)|\big)_-\Big]\nn,
	}
%\RvdH{$|\BP|$ vs $\BP$}
where the notation is as in Lemma \ref{lem-coupling-CM} and its proof, and $|\hatGraphn(m_n)|$ and $|\hatTreen(m_n)|$ denote the number of half-edges and individuals found up to the $m_n$th step of the exploration starting from $\Ver_2$, and $m_n=(b+1)\overline{m}_n$, while $x_+=\max\{0,x\}$ and $x_-=\max\{0,-x\}$. 

To bound these expectations, we adapt the proof of Lemma \ref{lem-coupling-CM} to our setting. We start with the first term in \eqref{prob-Cevent2-compl}, for which we use the exploration up to size $\overline{m}_n$ used in the proof of Lemma \ref{lem-coupling-CM}. We note that the only way that $|\hatGraphn(t+1)|-|\hatGraphn(t)|$ can be larger than $|\hatTreen(t+1)|-|\hatTreen(t)|$ is when a half-edge re-use occurs. This gives rise a {\em primary} ghost, which then possibly gives rise to some further {\em secondary} ghosts, i.e., ghost vertices that are found by pairing a ghost half-edge. Since all degrees are bounded by $b$, on $\Ccal_n(2,1)$ where the first $\underline{k}_n$ generations are perfectly coupled, the total number of secondary ghosts, together with the single primary ghost, is at most
	\eqn{
	c_n\equiv \sum_{k=0}^{\bar{k}_n-\underline{k}_n} (b-1)^k=\frac{(b-1)^{\bar{k}_n-\underline{k}_n+1}-1}{b-2},
	}
which tends to infinity. On $\Dcal_n$, we can let $\bar{k}_n-\underline{k}_n\rightarrow \infty$ arbitrarily slowly. Therefore, on $\Ccal_n(2,1)\cap \Dcal_n$,
	\eqn{
	|\hatGraphn(\overline{m}_n)|-|\hatTreen(\overline{m}_n)|\leq c_n \#\{\text{half-edge re-uses up to time }\overline{m}_n\}.
	}
Thus, by \eqref{half-edge-re-uses-CM},
	\eqn{
	\expec\Big[\indicwo{\Ccal_n(2,1)\cap \Dcal_n}\big(|\hatGraphn(\overline{m}_n)|-|\hatTreen(\overline{m}_n)|\big)_+\Big]\leq c_n\frac{m_n^2}{\ell_n}.
	}
We continue with the second term in \eqref{prob-Cevent2-compl}, which is similar. We note that $|\hatGraphn(t+1)|-|\hatGraphn(t)|$ can be smaller than $|\hatTreen(t+1)|-|\hatTreen(t)|$ when a half-edge re-use occurs, or when a vertex re-use occurs. Thus, again using that the total number of secondary ghosts, together with the single primary ghost, is at most $c_n$, on $\Ccal_n(2,1)\cap \Dcal_n$,
	\eqn{
	|\hatTreen(\overline{m}_n)|-|\hatGraphn(\overline{m}_n)|\leq c_n \#\{\text{half-edge and vertex re-uses up to time }\overline{m}_n\}.
	}
As a result, by \eqref{half-edge-re-uses-CM} and \eqref{vertex-re-uses-CM},
	\eqn{
	\expec\Big[\indicwo{\Ccal_n(2,1)\cap \Dcal_n}\big(|\hatGraphn(\overline{m}_n)|-|\hatTreen(\overline{m}_n)|\big)_+\Big]\leq c_n\frac{\overline{m}_n^2}{\ell_n}.
	}
We conclude that
	\eqn{
	\prob(\Ccal_n(2)^c)\leq \frac{1}{a_n} 2c_n\frac{\overline{m}_n^2}{\ell_n}=O\Big(\Big(\frac{\ell_n}{\overline{m}_n^2}\Big)^{\delta/2}\Big)=o(1),
	}
when taking $c_n\rightarrow \infty$ such that $c_n=o((\overline{m}_n^2/\ell_n)^{\delta/2})$.
\end{proof}

\end{appendix}

\paragraph{\bf Acknowledgements.}
The work in this paper was supported by the Netherlands Organisation for Scientific Research (NWO) through Gravitation-grant NETWORKS-024.002.003. I thank Joel Spencer for proposing the `easiest solution' to the giant component problem for the \erdos, which inspired this paper. I thank Shankar Bhamidi for useful and inspiring discussions, Souvik Dhara for literature references, Joost Jorritsma for pointing me to the relation to \cite{JorKomMit25}, and Christian Borgs for suggestions on local convergence and encouragement.

%%%%%%%%%%%%%%%%%%%%%%%%%%%%%%%%%%%%%%%%%%%%%%
%% Supplementary Material, if any, should   %%
%% be provided in {supplement} environment  %%
%% with title and short description.        %%
%%%%%%%%%%%%%%%%%%%%%%%%%%%%%%%%%%%%%%%%%%%%%%
%\begin{supplement}
%\stitle{???}
%\sdescription{???.}
%\end{supplement}

%%  MR numbers will be added by VTeX.                      %%
%%                                                         %%
%%  Use \cite{...} to cite references in text.             %%
%%                                                         %%
%%%%%%%%%%%%%%%%%%%%%%%%%%%%%%%%%%%%%%%%%%%%%%%%%%%%%%%%%%%%%

%% if your bibliography is in bibtex format, uncomment commands:
\bibliographystyle{imsart-number} % Style BST file (imsart-number.bst or imsart-nameyear.bst)
%\bibliography{../../../../../bib/bibBooks}       % Bibliography file (usually '*.bib')
\providecommand{\noopsort}[1]{}\def\cprime{$'$}

\end{document}

%% file: preamble.tex
\newcommand{\invisible}[1]{}

\newcommand{\citeI}[1]{\cite[#1]{Hofs17}}

\usepackage{amssymb}
\usepackage{dsfont}
\usepackage{graphicx}
\usepackage{bbm}
\usepackage{ifpdf}
\usepackage[textsize=small, color=magenta]{todonotes}
\usepackage{bbm}
\usepackage{fancyvrb}
\usepackage{mathrsfs}

\usepackage{xcolor}
\definecolor{aquamarine}{rgb}{0.5, 1.0, 0.83}
\definecolor{aqua}{rgb}{0.0, 1.0, 1.0}
\usepackage{subcaption}

\usepackage[pdftex]{thumbpdf}       %create thumbnails in the PDF
\graphicspath{{pictures/}}

\usepackage{eepic}
\usepackage{multirow}

\usepackage{pgfplots}
\pgfplotsset{compat=newest}

\usepackage{hhline}

\ifpdf
\else
  \usepackage{nohyperref}
  \usepackage[pageref]{backref}
\fi

\usepackage{enumitem}

\theoremstyle{plain}
\newtheorem{theorem}{Theorem}[section]
\newtheorem{corollary}[theorem]{Corollary}

\newtheorem{lemma}[theorem]{Lemma}

\newtheorem{cond}[theorem]{Condition}

\newtheorem{op}[theorem]{Open problem}

\newtheorem{remark}[theorem]{Remark}
\newcommand{\todotemp}[1]{}

\newtheoremstyle{maria}{}{}{}{0cm}{\itshape}{}{\newline}{}
\theoremstyle{maria}

\numberwithin{equation}{section}
\numberwithin{theorem}{section}
\usepackage{fancyhdr}

\makeatletter
\renewcommand{\section}{\@startsection{section}{1}{0mm}{1.5\baselineskip}{\baselineskip}{\normalfont\large\sc}}
\makeatother

\makeatletter
\renewcommand{\subsection}{\@startsection{subsection}{2}{0mm}{\baselineskip}{0.5\baselineskip}{\normalfont\sc}}
\makeatother

  \makeatletter
    \def\cleardoublepage{\clearpage\if@twoside \ifodd\c@page\else
        \hbox{}
        \thispagestyle{empty}
        \newpage
        \if@twocolumn\hbox{}\newpage\fi\fi\fi}
  \makeatother

\makeatletter
\newcommand{\bfdit}{\boldsymbol{d}}

\newcommand{\Nbold} {{\mathbb N}}

\newcommand{\Ccal}   {\mathcal{C}}
\newcommand{\Dcal}   {\mathcal{D}}

\newcommand{\Gcal}   {\mathcal{G}}

\newcommand{\erdos}{Erd\H{o}s-R\'enyi random graph}

\newcommand{\CMnd}{{\rm CM}_n(\bfdit)}
\newcommand{\CMndprime}{{\rm CM}_{n'}(\bfdit')}

\newcommand{\gdist}[1]{{\rm dist}_{\sss #1}}

\newcommand{\Ver}{o}

\newcommand{\dmax}{d_{\mathrm{max}}}

\newcommand{\Treen}{{\sf Tree}_n}
\newcommand{\hatTreen}{\widehat {\sf Tree}_n}
\newcommand{\Graphn}{{\sf Graph}_n}
\newcommand{\hatGraphn}{\widehat{\sf Graph}_n}

\newcommand {\convd}{\stackrel{d}{\longrightarrow}}
\newcommand {\convp}{\stackrel{\sss {\mathbb P}}{\longrightarrow}}
\newcommand {\convas}{\stackrel{a.s.}{\longrightarrow}}

\renewcommand{\op}{o_{\sss {\mathbb P}}}
\newcommand{\opk}{o_{k, {\sss {\mathbb P}}}}

\newcommand\1{\mathbbm{1}}
\newcommand{\indic}[1]{\1_{\{#1\}}}
\newcommand{\indicwo}[1]{\1_{#1}}

\newcommand{\expec}{{\mathbb{E}}}

\newcommand{\prob}{{\mathbb{P}}}

\newcommand{\Var}{\mathrm{Var}}

\newcommand{\N}{\Nbold}

\newcommand{\sss}   { \scriptscriptstyle }

\newcommand {\vep}{\varepsilon}

\newcommand{\conn}{\longleftrightarrow}

\newcommand{\nc}        { \conn {\hspace{-2.8ex} /} \hspace{1.8ex}   }

\newcommand{\cluster}{\mathscr{C}}
\newcommand{\Cmax}{\mathscr{C}_{\rm max}}

\newcommand{\BP}{{\sf BP}}

\def\eqalign#1\enalign{
    \begin{align}#1\end{align}
    }
\newcommand{\eq}{\begin{equation}}
\newcommand{\en}{\end{equation}}
\newcommand{\ben}{\begin{enumerate}}
\newcommand{\een}{\end{enumerate}}
\newcommand{\eqn}[1]{\begin{equation} #1 \end{equation}}
\newcommand{\eqan}[1]{\eqalign #1 \enalign}

\newcommand{\nn}{\nonumber}

\renewcommand{\to}{\rightarrow}
\newcommand{\ra}{\rightarrow}

\newcommand{\vertex}{o}

\newcommand{\beginsol}[1]{}

%% file: giant_almost_local_JAP.bbl
\begin{thebibliography}{73}
% BibTex style file: imsart-number.bst, 2017-11-03
% Default style options (sort=1,type=number).
% Used options (sort=1,type=number).

\bibitem{Aldo01}
\begin{barticle}[author]
\bauthor{\bsnm{Aldous},~\bfnm{D.}\binits{D.}}
(\byear{2001}).
\btitle{The {$\zeta(2)$} limit in the random assignment problem}.
\bjournal{Random Structures Algorithms}
\bvolume{{\bf 18}}
\bpages{381--418}.
\bdoi{10.1002/rsa.1015}
\bmrnumber{1839499}
\end{barticle}
\endbibitem

\bibitem{AldLyo07}
\begin{barticle}[author]
\bauthor{\bsnm{Aldous},~\bfnm{D.}\binits{D.}} \AND
  \bauthor{\bsnm{Lyons},~\bfnm{R.}\binits{R.}}
(\byear{2007}).
\btitle{Processes on unimodular random networks}.
\bjournal{Electron. J. Probab.}
\bvolume{{\bf 12}}
\bpages{1454--1508}.
\bdoi{10.1214/EJP.v12-463}
\bmrnumber{2354165}
\end{barticle}
\endbibitem

\bibitem{AldSte04}
\begin{bincollection}[author]
\bauthor{\bsnm{Aldous},~\bfnm{D.}\binits{D.}} \AND
  \bauthor{\bsnm{Steele},~\bfnm{J.~M.}\binits{J.~M.}}
(\byear{2004}).
\btitle{The objective method: probabilistic combinatorial optimization and
  local weak convergence}.
In \bbooktitle{Probability on discrete structures}.
\bseries{Encyclopaedia Math. Sci.}
\bvolume{{\bf 110}}
\bpages{1--72}.
\bpublisher{Springer}.
\bdoi{10.1007/978-3-662-09444-0_1}
\bmrnumber{2023650 (2005e:60018)}
\end{bincollection}
\endbibitem

\bibitem{AliBorSab23}
\begin{barticle}[author]
\bauthor{\bsnm{Alimohammadi},~\bfnm{Y.}\binits{Y.}},
  \bauthor{\bsnm{Borgs},~\bfnm{C.}\binits{C.}} \AND
  \bauthor{\bsnm{Saberi},~\bfnm{A.}\binits{A.}}
(\byear{2023}).
\btitle{Locality of random digraphs on expanders}.
\bjournal{Ann. Probab.}
\bvolume{{\bf 51}}
\bpages{1249--1297}.
\bdoi{10.1214/22-aop1618}
\bmrnumber{4597319}
\end{barticle}
\endbibitem

\bibitem{AloBenSta04}
\begin{barticle}[author]
\bauthor{\bsnm{Alon},~\bfnm{N.}\binits{N.}},
  \bauthor{\bsnm{Benjamini},~\bfnm{I.}\binits{I.}} \AND
  \bauthor{\bsnm{Stacey},~\bfnm{A.}\binits{A.}}
(\byear{2004}).
\btitle{Percolation on finite graphs and isoperimetric inequalities}.
\bjournal{Ann. Probab.}
\bvolume{{\bf 32}}
\bpages{1727--1745}.
\bmrnumber{MR2073175 (2005f:05149)}
\end{barticle}
\endbibitem

\bibitem{AloSpe00}
\begin{bbook}[author]
\bauthor{\bsnm{Alon},~\bfnm{N.}\binits{N.}} \AND
  \bauthor{\bsnm{Spencer},~\bfnm{J.}\binits{J.}}
(\byear{2000}).
\btitle{The probabilistic method},
\bedition{second} ed.
\bseries{Wiley-Interscience Series in Discrete Mathematics and Optimization}.
\bpublisher{John Wiley \& Sons}, \baddress{New York}.
\bmrnumber{MR1885388 (2003f:60003)}
\end{bbook}
\endbibitem

\bibitem{AnaSal16}
\begin{barticle}[author]
\bauthor{\bsnm{Anantharam},~\bfnm{V.}\binits{V.}} \AND
  \bauthor{\bsnm{Salez},~\bfnm{J.}\binits{J.}}
(\byear{2016}).
\btitle{The densest subgraph problem in sparse random graphs}.
\bjournal{Ann. Appl. Probab.}
\bvolume{{\bf 26}}
\bpages{305--327}.
\bdoi{10.1214/14-AAP1091}
\bmrnumber{3449319}
\end{barticle}
\endbibitem

\bibitem{BarAlb99}
\begin{barticle}[author]
\bauthor{\bsnm{Barab{\'a}si},~\bfnm{A.~L.}\binits{A.~L.}} \AND
  \bauthor{\bsnm{Albert},~\bfnm{R.}\binits{R.}}
(\byear{1999}).
\btitle{Emergence of scaling in random networks}.
\bjournal{Science}
\bvolume{{\bf 286}}
\bpages{509--512}.
\bmrnumber{MR2091634}
\end{barticle}
\endbibitem

\bibitem{BenBouLugRos12}
\begin{barticle}[author]
\bauthor{\bsnm{Benjamini},~\bfnm{I.}\binits{I.}},
  \bauthor{\bsnm{Boucheron},~\bfnm{S.}\binits{S.}},
  \bauthor{\bsnm{Lugosi},~\bfnm{G.}\binits{G.}} \AND
  \bauthor{\bsnm{Rossignol},~\bfnm{R.}\binits{R.}}
(\byear{2012}).
\btitle{Sharp threshold for percolation on expanders}.
\bjournal{Ann. Probab.}
\bvolume{{\bf 40}}
\bpages{130--145}.
\bdoi{10.1214/10-AOP610}
\bmrnumber{2917769}
\end{barticle}
\endbibitem

\bibitem{BenNacPer11}
\begin{barticle}[author]
\bauthor{\bsnm{Benjamini},~\bfnm{I.}\binits{I.}},
  \bauthor{\bsnm{Nachmias},~\bfnm{A.}\binits{A.}} \AND
  \bauthor{\bsnm{Peres},~\bfnm{Y.}\binits{Y.}}
(\byear{2011}).
\btitle{Is the critical percolation probability local?}
\bjournal{Probab. Theory Rel. Fields}
\bvolume{{\bf 149}}
\bpages{261--269}.
\bdoi{10.1007/s00440-009-0251-5}
\bmrnumber{2773031}
\end{barticle}
\endbibitem

\bibitem{BenSch01}
\begin{barticle}[author]
\bauthor{\bsnm{Benjamini},~\bfnm{I.}\binits{I.}} \AND
  \bauthor{\bsnm{Schramm},~\bfnm{O.}\binits{O.}}
(\byear{2001}).
\btitle{Recurrence of distributional limits of finite planar graphs}.
\bjournal{Electron. J. Probab.}
\bvolume{{\bf 6}}
\bpages{13 pp. (electronic)}.
\bdoi{10.1214/EJP.v6-96}
\bmrnumber{1873300}
\end{barticle}
\endbibitem

\bibitem{BerBorChaSab14}
\begin{barticle}[author]
\bauthor{\bsnm{Berger},~\bfnm{N.}\binits{N.}},
  \bauthor{\bsnm{Borgs},~\bfnm{C.}\binits{C.}},
  \bauthor{\bsnm{Chayes},~\bfnm{J.}\binits{J.}} \AND
  \bauthor{\bsnm{Saberi},~\bfnm{A.}\binits{A.}}
(\byear{2014}).
\btitle{Asymptotic behavior and distributional limits of preferential
  attachment graphs}.
\bjournal{Ann. Probab.}
\bvolume{{\bf 42}}
\bpages{1--40}.
\bdoi{10.1214/12-AOP755}
\bmrnumber{3161480}
\end{barticle}
\endbibitem

\bibitem{BhaEvaSen12}
\begin{barticle}[author]
\bauthor{\bsnm{Bhamidi},~\bfnm{S.}\binits{S.}},
  \bauthor{\bsnm{Evans},~\bfnm{S.}\binits{S.}} \AND
  \bauthor{\bsnm{Sen},~\bfnm{A.}\binits{A.}}
(\byear{2012}).
\btitle{Spectra of large random trees}.
\bjournal{J. Theoret. Probab.}
\bvolume{{\bf 25}}
\bpages{613--654}.
\bdoi{10.1007/s10959-011-0360-9}
\bmrnumber{2956206}
\end{barticle}
\endbibitem

\bibitem{BhaHofHoo17}
\begin{barticle}[author]
\bauthor{\bsnm{Bhamidi},~\bfnm{S.}\binits{S.}},
  \bauthor{\bsnm{{\noopsort{Hofstad}{van der Hofstad}}},~\bfnm{R.}\binits{R.}}
  \AND \bauthor{\bsnm{Hooghiemstra},~\bfnm{G.}\binits{G.}}
(\byear{2017}).
\btitle{Universality for first passage percolation on sparse random graphs}.
\bjournal{Ann. Probab.}
\bvolume{{\bf 45}}
\bpages{2568--2630}.
\bdoi{10.1214/16-AOP1120}
\bmrnumber{3693970}
\end{barticle}
\endbibitem

\bibitem{BodFouMul15}
\begin{barticle}[author]
\bauthor{\bsnm{Bode},~\bfnm{M.}\binits{M.}},
  \bauthor{\bsnm{Fountoulakis},~\bfnm{N.}\binits{N.}} \AND
  \bauthor{\bsnm{M\"uller},~\bfnm{T.}\binits{T.}}
(\byear{2015}).
\btitle{On the largest component of a hyperbolic model of complex networks}.
\bjournal{Electron. J. Combin.}
\bvolume{22}
\bpages{Paper 3.24, 46}.
\bmrnumber{3386525}
\end{barticle}
\endbibitem

\bibitem{Boll80b}
\begin{barticle}[author]
\bauthor{\bsnm{Bollob{\'a}s},~\bfnm{B.}\binits{B.}}
(\byear{1980}).
\btitle{A probabilistic proof of an asymptotic formula for the number of
  labelled regular graphs}.
\bjournal{European J. Combin.}
\bvolume{{\bf 1}}
\bpages{311--316}.
\bmrnumber{MR595929 (82i:05045)}
\end{barticle}
\endbibitem

\bibitem{Boll01}
\begin{bbook}[author]
\bauthor{\bsnm{Bollob{\'a}s},~\bfnm{B.}\binits{B.}}
(\byear{2001}).
\btitle{Random graphs},
\bedition{second} ed.
\bseries{Cambridge Studies in Advanced Mathematics}
\bvolume{{\bf 73}}.
\bpublisher{Cambridge University Press}.
\bmrnumber{MR1864966 (2002j:05132)}
\end{bbook}
\endbibitem

\bibitem{BolJanRio07}
\begin{barticle}[author]
\bauthor{\bsnm{Bollob{\'a}s},~\bfnm{B.}\binits{B.}},
  \bauthor{\bsnm{Janson},~\bfnm{S.}\binits{S.}} \AND
  \bauthor{\bsnm{Riordan},~\bfnm{O.}\binits{O.}}
(\byear{2007}).
\btitle{The phase transition in inhomogeneous random graphs}.
\bjournal{Random Structures Algorithms}
\bvolume{{\bf 31}}
\bpages{3--122}.
\bmrnumber{MR2337396}
\end{barticle}
\endbibitem

\bibitem{BolRio15}
\begin{barticle}[author]
\bauthor{\bsnm{Bollob\'as},~\bfnm{B.}\binits{B.}} \AND
  \bauthor{\bsnm{Riordan},~\bfnm{O.}\binits{O.}}
(\byear{2015}).
\btitle{An old approach to the giant component problem}.
\bjournal{J. Combin. Theory Ser. B}
\bvolume{{\bf 113}}
\bpages{236--260}.
\bmrnumber{3343756}
\end{barticle}
\endbibitem

\bibitem{BolRioSpeTus01}
\begin{barticle}[author]
\bauthor{\bsnm{Bollob{\'a}s},~\bfnm{B.}\binits{B.}},
  \bauthor{\bsnm{Riordan},~\bfnm{O.}\binits{O.}},
  \bauthor{\bsnm{Spencer},~\bfnm{J.}\binits{J.}} \AND
  \bauthor{\bsnm{Tusn{\'a}dy},~\bfnm{G.}\binits{G.}}
(\byear{2001}).
\btitle{The degree sequence of a scale-free random graph process}.
\bjournal{Random Structures Algorithms}
\bvolume{{\bf 18}}
\bpages{279--290}.
\bmrnumber{MR1824277 (2002b:05121)}
\end{barticle}
\endbibitem

\bibitem{Bord16}
\begin{bunpublished}[author]
\bauthor{\bsnm{Bordenave},~\bfnm{C.}\binits{C.}}
(\byear{2016}).
\btitle{Lecture notes on random graphs and probabilistic combinatorial
  optimization}.
\bnote{Version April 8, 2016. Available at
  \url{www.math.univ-toulouse.fr/~bordenave/coursRG.pdf}}.
\end{bunpublished}
\endbibitem

\bibitem{BorCap15}
\begin{barticle}[author]
\bauthor{\bsnm{Bordenave},~\bfnm{C.}\binits{C.}} \AND
  \bauthor{\bsnm{Caputo},~\bfnm{P.}\binits{P.}}
(\byear{2015}).
\btitle{Large deviations of empirical neighborhood distribution in sparse
  random graphs}.
\bjournal{Probab. Theory Rel. Fields}
\bvolume{{\bf 163}}
\bpages{149--222}.
\bdoi{10.1007/s00440-014-0590-8}
\bmrnumber{3405616}
\end{barticle}
\endbibitem

\bibitem{BorChaHofSlaSpe05a}
\begin{barticle}[author]
\bauthor{\bsnm{Borgs},~\bfnm{C.}\binits{C.}},
  \bauthor{\bsnm{Chayes},~\bfnm{J.}\binits{J.}},
  \bauthor{\bsnm{{\noopsort{Hofstad}{van der Hofstad}}},~\bfnm{R.}\binits{R.}},
  \bauthor{\bsnm{Slade},~\bfnm{G.}\binits{G.}} \AND
  \bauthor{\bsnm{Spencer},~\bfnm{J.}\binits{J.}}
(\byear{2005}).
\btitle{Random subgraphs of finite graphs. {I}. {T}he scaling window under the
  triangle condition}.
\bjournal{Random Structures Algorithms}
\bvolume{{\bf 27}}
\bpages{137--184}.
\bmrnumber{MR2155704 (2006e:05156)}
\end{barticle}
\endbibitem

\bibitem{BorChaHofSlaSpe06}
\begin{barticle}[author]
\bauthor{\bsnm{Borgs},~\bfnm{C.}\binits{C.}},
  \bauthor{\bsnm{Chayes},~\bfnm{J.}\binits{J.}},
  \bauthor{\bsnm{{\noopsort{Hofstad}{van der Hofstad}}},~\bfnm{R.}\binits{R.}},
  \bauthor{\bsnm{Slade},~\bfnm{G.}\binits{G.}} \AND
  \bauthor{\bsnm{Spencer},~\bfnm{J.}\binits{J.}}
(\byear{2006}).
\btitle{Random subgraphs of finite graphs. {III}. {The} phase transition for
  the $n$-cube}.
\bjournal{Combinatorica}
\bvolume{{\bf 26}}
\bpages{395--410}.
\bmrnumber{MR2260845}
\end{barticle}
\endbibitem

\bibitem{BorChaKesSpe99}
\begin{barticle}[author]
\bauthor{\bsnm{Borgs},~\bfnm{C.}\binits{C.}},
  \bauthor{\bsnm{Chayes},~\bfnm{J.~T.}\binits{J.~T.}},
  \bauthor{\bsnm{Kesten},~\bfnm{H.}\binits{H.}} \AND
  \bauthor{\bsnm{Spencer},~\bfnm{J.}\binits{J.}}
(\byear{1999}).
\btitle{Uniform boundedness of critical crossing probabilities implies
  hyperscaling}.
\bjournal{Random Structures Algorithms}
\bvolume{{\bf 15}}
\bpages{368--413}.
\bmrnumber{MR1716769 (2001a:60111)}
\end{barticle}
\endbibitem

\bibitem{BriKeuLen16}
\begin{bunpublished}[author]
\bauthor{\bsnm{Bringmann},~\bfnm{K.}\binits{K.}},
  \bauthor{\bsnm{Keusch},~\bfnm{R.}\binits{R.}} \AND
  \bauthor{\bsnm{Lengler},~\bfnm{J.}\binits{J.}}
(\byear{2020}).
\btitle{Average distance in a general class of scale-free networks with
  underlying geometry}.
\bnote{ar{X}iv: 1602.05712 [cs.{DM}]}.
\end{bunpublished}
\endbibitem

\bibitem{ChuLu06c}
\begin{bbook}[author]
\bauthor{\bsnm{Chung},~\bfnm{F.}\binits{F.}} \AND
  \bauthor{\bsnm{Lu},~\bfnm{L.}\binits{L.}}
(\byear{2006}).
\btitle{Complex graphs and networks}.
\bseries{CBMS Regional Conference Series in Mathematics}
\bvolume{{\bf 107}}.
\bpublisher{Conference Board of the Mathematical Sciences, Washington, DC; by
  the American Mathematical Society, Providence, RI}.
\bdoi{10.1090/cbms/107}
\bmrnumber{MR2248695 (2007i:05169)}
\end{bbook}
\endbibitem

\bibitem{DemMon10a}
\begin{barticle}[author]
\bauthor{\bsnm{Dembo},~\bfnm{A.}\binits{A.}} \AND
  \bauthor{\bsnm{Montanari},~\bfnm{A.}\binits{A.}}
(\byear{2010}).
\btitle{{Ising models on locally tree-like graphs}}.
\bjournal{Ann. Appl. Probab.}
\bvolume{{\bf 20}}
\bpages{565--592}.
\bdoi{10.1214/09-AAP627}
\bmrnumber{2650042}
\end{barticle}
\endbibitem

\bibitem{DemMon10b}
\begin{barticle}[author]
\bauthor{\bsnm{Dembo},~\bfnm{A.}\binits{A.}} \AND
  \bauthor{\bsnm{Montanari},~\bfnm{A.}\binits{A.}}
(\byear{2010}).
\btitle{Gibbs measures and phase transitions on sparse random graphs}.
\bjournal{Braz. J. Probab. Statist.}
\bvolume{{\bf 24}}
\bpages{137--211}.
\bdoi{10.1214/09-BJPS027}
\bmrnumber{2643563}
\end{barticle}
\endbibitem

\bibitem{DerMonMor12}
\begin{barticle}[author]
\bauthor{\bsnm{Dereich},~\bfnm{S.}\binits{S.}},
  \bauthor{\bsnm{M{\"o}nch},~\bfnm{C.}\binits{C.}} \AND
  \bauthor{\bsnm{M{\"o}rters},~\bfnm{P.}\binits{P.}}
(\byear{2012}).
\btitle{Typical distances in ultrasmall random networks}.
\bjournal{Adv. Appl. Probab.}
\bvolume{{\bf 44}}
\bpages{583--601}.
\bdoi{10.1239/aap/1339878725}
\bmrnumber{2977409}
\end{barticle}
\endbibitem

\bibitem{DerMonMor17}
\begin{barticle}[author]
\bauthor{\bsnm{Dereich},~\bfnm{S.}\binits{S.}},
  \bauthor{\bsnm{M{\"o}nch},~\bfnm{C.}\binits{C.}} \AND
  \bauthor{\bsnm{M{\"o}rters},~\bfnm{P.}\binits{P.}}
(\byear{2017}).
\btitle{Distances in scale free networks at criticality}.
\bjournal{Electron. J. Probab.}
\bvolume{{\bf 22}}
\bpages{Paper No. 77, 38}.
\bdoi{10.1214/17-EJP92}
\bmrnumber{3710797}
\end{barticle}
\endbibitem

\bibitem{DerMor13}
\begin{barticle}[author]
\bauthor{\bsnm{Dereich},~\bfnm{S.}\binits{S.}} \AND
  \bauthor{\bsnm{M{\"o}rters},~\bfnm{P.}\binits{P.}}
(\byear{2013}).
\btitle{Random networks with sublinear preferential attachment: the giant
  component}.
\bjournal{Ann. Probab.}
\bvolume{{\bf 41}}
\bpages{329--384}.
\bdoi{10.1214/11-AOP697}
\bmrnumber{3059201}
\end{barticle}
\endbibitem

\bibitem{Durr06}
\begin{bbook}[author]
\bauthor{\bsnm{Durrett},~\bfnm{R.}\binits{R.}}
(\byear{2007}).
\btitle{Random graph dynamics}.
\bseries{Cambridge Series in Statistical and Probabilistic Mathematics}.
\bpublisher{Cambridge University Press}.
\bmrnumber{MR2271734}
\end{bbook}
\endbibitem

\bibitem{ErdRen60}
\begin{barticle}[author]
\bauthor{\bsnm{Erd{\H{o}}s},~\bfnm{P.}\binits{P.}} \AND
  \bauthor{\bsnm{R{\'e}nyi},~\bfnm{A.}\binits{A.}}
(\byear{1960}).
\btitle{On the evolution of random graphs}.
\bjournal{Magyar Tud. Akad. Mat. Kutat{\'o} Int. K{\"o}zl.}
\bvolume{{\bf 5}}
\bpages{17--61}.
\bmrnumber{MR0125031 (23 \#\#A2338)}
\end{barticle}
\endbibitem

\bibitem{Fede20}
\begin{barticle}[author]
\bauthor{\bsnm{Federico},~\bfnm{L}\binits{L.}}
(\byear{2023}).
\btitle{Almost-2-regular random graphs}.
\bjournal{Australas. J. Combin.}
\bvolume{86}
\bpages{76--96}.
\bdoi{10.1007/s10998-022-00461-x}
\bmrnumber{4587725}
\end{barticle}
\endbibitem

\bibitem{Foun07}
\begin{barticle}[author]
\bauthor{\bsnm{Fountoulakis},~\bfnm{N.}\binits{N.}}
(\byear{2007}).
\btitle{Percolation on sparse random graphs with given degree sequence}.
\bjournal{Internet Math.}
\bvolume{{\bf 4}}
\bpages{329--356}.
\bmrnumber{2522948}
\end{barticle}
\endbibitem

\bibitem{FouMul18}
\begin{barticle}[author]
\bauthor{\bsnm{Fountoulakis},~\bfnm{N.}\binits{N.}} \AND
  \bauthor{\bsnm{M\"uller},~\bfnm{T.}\binits{T.}}
(\byear{2018}).
\btitle{Law of large numbers for the largest component in a hyperbolic model of
  complex networks}.
\bjournal{Ann. Appl. Probab.}
\bvolume{{\bf 28}}
\bpages{607--650}.
\bdoi{10.1214/17-AAP1314}
\bmrnumber{3770885}
\end{barticle}
\endbibitem

\bibitem{FriKar16}
\begin{bbook}[author]
\bauthor{\bsnm{Frieze},~\bfnm{A.}\binits{A.}} \AND
  \bauthor{\bsnm{Karo\'{n}ski},~\bfnm{M.}\binits{M.}}
(\byear{2016}).
\btitle{Introduction to random graphs}.
\bpublisher{Cambridge University Press, Cambridge}.
\bdoi{10.1017/CBO9781316339831}
\bmrnumber{3675279}
\end{bbook}
\endbibitem

\bibitem{FriKar23}
\begin{bbook}[author]
\bauthor{\bsnm{Frieze},~\bfnm{A.}\binits{A.}} \AND
  \bauthor{\bsnm{Karo\'{n}ski},~\bfnm{M.}\binits{M.}}
(\byear{2023}).
\btitle{Random graphs and networks: a first course}.
\bpublisher{Cambridge University Press, Cambridge}.
\bdoi{10.1017/9781009260268}
\bmrnumber{4567944}
\end{bbook}
\endbibitem

\bibitem{GamNowSwi06}
\begin{barticle}[author]
\bauthor{\bsnm{Gamarnik},~\bfnm{D.}\binits{D.}},
  \bauthor{\bsnm{Nowicki},~\bfnm{T.}\binits{T.}} \AND
  \bauthor{\bsnm{Swirszcz},~\bfnm{G.}\binits{G.}}
(\byear{2006}).
\btitle{Maximum weight independent sets and matchings in sparse random graphs.
  {E}xact results using the local weak convergence method}.
\bjournal{Random Structures Algorithms}
\bvolume{{\bf 28}}
\bpages{76--106}.
\bdoi{10.1002/rsa.20072}
\bmrnumber{2187483}
\end{barticle}
\endbibitem

\bibitem{GarHazHofRay22}
\begin{bunpublished}[author]
\bauthor{\bsnm{Garavaglia},~\bfnm{A.}\binits{A.}},
  \bauthor{\bsnm{Hazra},~\bfnm{R.}\binits{R.}},
  \bauthor{\bsnm{{\noopsort{Hofstad}{van der Hofstad}}},~\bfnm{R.}\binits{R.}}
  \AND \bauthor{\bsnm{Ray},~\bfnm{R.}\binits{R.}}
(\byear{2022}).
\btitle{Universality of the local limit in preferential attachment models}.
\bnote{arXiv:2212.05551 [math.PR]}.
\end{bunpublished}
\endbibitem

\bibitem{GarHofLit20}
\begin{barticle}[author]
\bauthor{\bsnm{Garavaglia},~\bfnm{A.}\binits{A.}},
  \bauthor{\bsnm{{\noopsort{Hofstad}{van der Hofstad}}},~\bfnm{R.}\binits{R.}}
  \AND \bauthor{\bsnm{Litvak},~\bfnm{N.}\binits{N.}}
(\byear{2020}).
\btitle{Local weak convergence for {P}age{R}ank}.
\bjournal{Ann. Appl. Probab.}
\bvolume{{\bf 30}}
\bpages{40--79}.
\bdoi{10.1214/19-AAP1494}
\bmrnumber{4068306}
\end{barticle}
\endbibitem


\bibitem{Hofs09a}
\begin{barticle}[author]
\bauthor{\bsnm{{\noopsort{Hofstad}{van der Hofstad}}},~\bfnm{R.}\binits{R.}}
(\byear{2013}).
\btitle{Critical behavior in inhomogeneous random graphs}.
\bjournal{Random Structures and Algorithms}
\bvolume{{\bf 42}}
\bpages{480--508}.
\bdoi{10.1002/rsa.20450}
\bmrnumber{3068034}
\end{barticle}
\endbibitem

\bibitem{Hofs17}
\begin{bbook}[author]
\bauthor{\bsnm{{\noopsort{Hofstad}{van der Hofstad}}},~\bfnm{R.}\binits{R.}}
(\byear{2017}).
\btitle{Random graphs and complex networks. {V}olume 1}.
\bseries{Cambridge Series in Statistical and Probabilistic Mathematics}
\bvolume{{\bf 43}}.
\bpublisher{Cambridge University Press}.
\bdoi{10.1017/9781316779422}
\bmrnumber{3617364}
\end{bbook}
\endbibitem

\bibitem{Hofs24}
\begin{bbook}[author]
\bauthor{\bsnm{{\noopsort{Hofstad}{van der Hofstad}}},~\bfnm{R.}\binits{R.}}
(\byear{2024}).
\btitle{Random graphs and complex networks. {V}olume 2}.
\bseries{Cambridge Series in Statistical and Probabilistic Mathematics}
\bvolume{{\bf 54}}.
\bpublisher{Cambridge University Press}.
\bdoi{10.1017/9781316795552}
%\bmrnumber{3617364}
\end{bbook}
\endbibitem

\bibitem{HofHooVan05a}
\begin{barticle}[author]
\bauthor{\bsnm{{\noopsort{Hofstad}{van der Hofstad}}},~\bfnm{R.}\binits{R.}},
  \bauthor{\bsnm{Hooghiemstra},~\bfnm{G.}\binits{G.}} \AND
  \bauthor{\bsnm{Van~Mieghem},~\bfnm{P.}\binits{P.}}
(\byear{2005}).
\btitle{Distances in random graphs with finite variance degrees}.
\bjournal{Random Structures Algorithms}
\bvolume{{\bf 27}}
\bpages{76--123}.
\bmrnumber{MR2150017 (2006d:60019)}
\end{barticle}
\endbibitem

\bibitem{HofHooZna04a}
\begin{barticle}[author]
\bauthor{\bsnm{{\noopsort{Hofstad}{van der Hofstad}}},~\bfnm{R.}\binits{R.}},
  \bauthor{\bsnm{Hooghiemstra},~\bfnm{G.}\binits{G.}} \AND
  \bauthor{\bsnm{Znamenski},~\bfnm{D.}\binits{D.}}
(\byear{2007}).
\btitle{Distances in random graphs with finite mean and infinite variance
  degrees}.
\bjournal{Electron. J. Probab.}
\bvolume{{\bf 12}}
\bpages{703--766 (electronic)}.
\bmrnumber{MR2318408}
\end{barticle}
\endbibitem

\bibitem{HofKom17}
\begin{barticle}[author]
\bauthor{\bsnm{{\noopsort{Hofstad}{van der Hofstad}}},~\bfnm{R.}\binits{R.}}
  \AND \bauthor{\bsnm{Komj\'{a}thy},~\bfnm{J.}\binits{J.}}
(\byear{2017}).
\btitle{When is a scale-free graph ultra-small?}
\bjournal{J. Statist. Phys.}
\bvolume{{\bf 169}}
\bpages{223--264}.
\bdoi{10.1007/s10955-017-1864-1}
\bmrnumber{3704860}
\end{barticle}
\endbibitem

\bibitem{HofKomVad18}
\begin{barticle}[author]
\bauthor{\bsnm{{\noopsort{Hofstad}{van der Hofstad}}},~\bfnm{R.}\binits{R.}},
  \bauthor{\bsnm{Komj{\'a}thy},~\bfnm{J.}\binits{J.}} \AND
  \bauthor{\bsnm{Vadon},~\bfnm{V.}\binits{V.}}
(\byear{2021}).
\btitle{Random intersection graphs with communities}.
\bjournal{Adv. Appl. Probab.}
\bvolume{53}
\bpages{1061--1089}.
\bdoi{10.1017/apr.2021.12}
\bmrnumber{4342577}
\end{barticle}
\endbibitem

\bibitem{HofNac13}
\begin{barticle}[author]
\bauthor{\bsnm{{\noopsort{Hofstad}{van der Hofstad}}},~\bfnm{R.}\binits{R.}}
  \AND \bauthor{\bsnm{Nachmias},~\bfnm{A.}\binits{A.}}
(\byear{2014}).
\btitle{Unlacing hypercube percolation: a survey}.
\bjournal{Metrika}
\bvolume{77}
\bpages{23--50}.
\bdoi{10.1007/s00184-013-0473-5}
\bmrnumber{3152019}
\end{barticle}
\endbibitem

\bibitem{HofNac17}
\begin{barticle}[author]
\bauthor{\bsnm{{\noopsort{Hofstad}{van der Hofstad}}},~\bfnm{R.}\binits{R.}}
  \AND \bauthor{\bsnm{Nachmias},~\bfnm{A.}\binits{A.}}
(\byear{2017}).
\btitle{Hypercube percolation}.
\bjournal{J. Eur. Math. Soc. (JEMS)}
\bvolume{{\bf 19}}
\bpages{725--814}.
\bdoi{10.4171/JEMS/679}
\bmrnumber{3612867}
\end{barticle}
\endbibitem

\bibitem{HofHooMai21}
\begin{barticle}[author]
\bauthor{\bsnm{{\noopsort{Hofstad}{van der Hofstad}}},~\bfnm{R.}\binits{R.}},
  \bauthor{\bsnm{{\noopsort{Hoorn}{van der Hoorn}}},~\bfnm{P.}\binits{P.}} \AND
  \bauthor{\bsnm{Maitra},~\bfnm{N.}\binits{N.}}
(\byear{2023}).
\btitle{Local limits of spatial inhomogeneous random graphs}.
\bjournal{Adv. in Appl. Probab.}
\bvolume{{\bf 55}}
\bpages{793--840}.
\bdoi{10.1017/apr.2022.61}
\bmrnumber{4624029}
\end{barticle}
\endbibitem

\bibitem{HofPan25}
\begin{barticle}[author]
\bauthor{\bsnm{{\noopsort{Hofstad}{van der Hofstad}}},~\bfnm{R.}\binits{R.}}
  \AND \bauthor{\bsnm{Pandey},~\bfnm{M.}\binits{M.}}
(\byear{2025}).
\btitle{Are giants in random digraphs `almost' local?}
\bjournal{Electron. Commun. Probab.}
\bvolume{{\bf 30}}
\bpages{1--13}.
\bdoi{10.1214/25-ECP694}
\end{barticle}
\endbibitem


\bibitem{Jans09c}
\begin{barticle}[author]
\bauthor{\bsnm{Janson},~\bfnm{S.}\binits{S.}}
(\byear{2009}).
\btitle{On percolation in random graphs with given vertex degrees}.
\bjournal{Electron. J. Probab.}
\bvolume{{\bf 14}}
\bpages{no. 5, 87--118}.
\bdoi{10.1214/EJP.v14-603}
\bmrnumber{2471661 (2010b:60023)}
\end{barticle}
\endbibitem

\bibitem{Jans09b}
\begin{barticle}[author]
\bauthor{\bsnm{Janson},~\bfnm{S.}\binits{S.}}
(\byear{2010}).
\btitle{Susceptibility of random graphs with given vertex degrees}.
\bjournal{J. Combin.}
\bvolume{{\bf 1}}
\bpages{357--387}.
\bmrnumber{2799217 (2012c:05278)}
\end{barticle}
\endbibitem

\bibitem{JanKnuLucPit93}
\begin{barticle}[author]
\bauthor{\bsnm{Janson},~\bfnm{S.}\binits{S.}},
  \bauthor{\bsnm{Knuth},~\bfnm{D.~E.}\binits{D.~E.}},
  \bauthor{\bsnm{{\L}uczak},~\bfnm{T.}\binits{T.}} \AND
  \bauthor{\bsnm{Pittel},~\bfnm{B.}\binits{B.}}
(\byear{1993}).
\btitle{The birth of the giant component}.
\bjournal{Random Structures Algorithms}
\bvolume{{\bf 4}}
\bpages{231--358}.
\bnote{With an introduction by the editors}.
\bmrnumber{MR1220220 (94h:05070)}
\end{barticle}
\endbibitem

\bibitem{JanLuc09}
\begin{barticle}[author]
\bauthor{\bsnm{Janson},~\bfnm{S.}\binits{S.}} \AND
  \bauthor{\bsnm{Luczak},~\bfnm{M.}\binits{M.}}
(\byear{2009}).
\btitle{A new approach to the giant component problem}.
\bjournal{Random Structures Algorithms}
\bvolume{{\bf 34}}
\bpages{197--216}.
\bdoi{10.1002/rsa.20231}
\bmrnumber{MR2490288 (2010d:05140)}
\end{barticle}
\endbibitem

\bibitem{JanLucRuc00}
\begin{bbook}[author]
\bauthor{\bsnm{Janson},~\bfnm{S.}\binits{S.}},
  \bauthor{\bsnm{{\L}uczak},~\bfnm{T.}\binits{T.}} \AND
  \bauthor{\bsnm{Rucinski},~\bfnm{A.}\binits{A.}}
(\byear{2000}).
\btitle{Random graphs}.
\bseries{Wiley-Interscience Series in Discrete Mathematics and Optimization}.
\bpublisher{Wiley-Interscience}.
\bmrnumber{MR1782847 (2001k:05180)}
\end{bbook}
\endbibitem

\bibitem{JanRio12}
\begin{barticle}[author]
\bauthor{\bsnm{Janson},~\bfnm{S.}\binits{S.}} \AND
  \bauthor{\bsnm{Riordan},~\bfnm{O.}\binits{O.}}
(\byear{2012}).
\btitle{Susceptibility in inhomogeneous random graphs}.
\bjournal{Electron. J. Combin.}
\bvolume{{\bf 19}}
\bpages{Paper 31, 59}.
\bmrnumber{2880662}
\end{barticle}
\endbibitem

\bibitem{JanWar18}
\begin{barticle}[author]
\bauthor{\bsnm{Janson},~\bfnm{S.}\binits{S.}} \AND
  \bauthor{\bsnm{Warnke},~\bfnm{L.}\binits{L.}}
(\byear{2018}).
\btitle{On the critical probability in percolation}.
\bjournal{Electron. J. Probab.}
\bvolume{{\bf 23}}
\bpages{Paper No. 1, 25}.
\bdoi{10.1214/17-EJP52}
\bmrnumber{3751076}
\end{barticle}
\endbibitem

\bibitem{JorKomMit24}
\begin{barticle}[author]
\bauthor{\bsnm{Jorritsma},~\bfnm{J.}\binits{J.}},
  \bauthor{\bsnm{Komj\'{a}thy},~\bfnm{J.}\binits{J.}} \AND
  \bauthor{\bsnm{Mitsche},~\bfnm{D.}\binits{D.}}
(\byear{2024}).
\btitle{Cluster-size decay in supercritical long-range percolation}.
\bjournal{Electron. J. Probab.}
\bvolume{{\bf 29}}
\bpages{Paper No. 82, 36}.
\bdoi{10.1214/24-ejp1135}
\bmrnumber{4758370}
\end{barticle}
\endbibitem

\bibitem{JorKomMit25}
\begin{barticle}[author]
\bauthor{\bsnm{Jorritsma},~\bfnm{J.}\binits{J.}},
  \bauthor{\bsnm{Komj\'{a}thy},~\bfnm{J.}\binits{J.}} \AND
  \bauthor{\bsnm{Mitsche},~\bfnm{D.}\binits{D.}}
(\byear{2025}).
\btitle{Cluster-size decay in supercritical kernel-based spatial random
  graphs}.
\bjournal{Ann. Probab.}
\bvolume{{\bf 53}}
\bpages{1537--1597}.
\bdoi{10.1214/24-aop1742}
\bmrnumber{4930534}
\end{barticle}
\endbibitem

\bibitem{KiwMit19}
\begin{barticle}[author]
\bauthor{\bsnm{Kiwi},~\bfnm{M.}\binits{M.}} \AND
  \bauthor{\bsnm{Mitsche},~\bfnm{D.}\binits{D.}}
(\byear{2019}).
\btitle{On the second largest component of random hyperbolic graphs}.
\bjournal{SIAM J. Discrete Math.}
\bvolume{{\bf 33}}
\bpages{2200--2217}.
\bdoi{10.1137/18M121201X}
\bmrnumber{4032854}
\end{barticle}
\endbibitem

\bibitem{KomLod20}
\begin{barticle}[author]
\bauthor{\bsnm{Komj\'{a}thy},~\bfnm{J\.}\binits{J.}} \AND
  \bauthor{\bsnm{Lodewijks},~\bfnm{B.}\binits{B.}}
(\byear{2020}).
\btitle{Explosion in weighted hyperbolic random graphs and geometric
  inhomogeneous random graphs}.
\bjournal{Stochastic Process. Appl.}
\bvolume{{\bf 130}}
\bpages{1309--1367}.
\bdoi{10.1016/j.spa.2019.04.014}
\bmrnumber{4058275}
\end{barticle}
\endbibitem

\bibitem{KriPapKitVahBog10}
\begin{barticle}[author]
\bauthor{\bsnm{Krioukov},~\bfnm{D.}\binits{D.}},
  \bauthor{\bsnm{Papadopoulos},~\bfnm{F.}\binits{F.}},
  \bauthor{\bsnm{Kitsak},~\bfnm{M.}\binits{M.}},
  \bauthor{\bsnm{Vahdat},~\bfnm{A.}\binits{A.}} \AND
  \bauthor{\bsnm{Bogu\~n\'a},~\bfnm{M.}\binits{M.}}
(\byear{2010}).
\btitle{Hyperbolic geometry of complex networks}.
\bjournal{Phys. Rev. E}
\bvolume{82}
\bpages{036106, 18}.
\bmrnumber{2787998}
\end{barticle}
\endbibitem

\bibitem{Kura15}
\begin{barticle}[author]
\bauthor{\bsnm{Kurauskas},~\bfnm{V.}\binits{V.}}
(\byear{2022}).
\btitle{On local weak limit and subgraph counts for sparse random graphs}.
\bjournal{J. Appl. Probab.}
\bvolume{{\bf 59}}
\bpages{755-776}.
\bdoi{10.1017/jpr.2021.84}
\bmrnumber{4480079}
\end{barticle}
\endbibitem

\bibitem{Lyon05}
\begin{barticle}[author]
\bauthor{\bsnm{Lyons},~\bfnm{R.}\binits{R.}}
(\byear{2005}).
\btitle{Asymptotic enumeration of spanning trees}.
\bjournal{Combin. Probab. Comput.}
\bvolume{{\bf 14}}
\bpages{491--522}.
\bdoi{10.1017/S096354830500684X}
\bmrnumber{2160416}
\end{barticle}
\endbibitem

\bibitem{MolRee95}
\begin{barticle}[author]
\bauthor{\bsnm{Molloy},~\bfnm{M.}\binits{M.}} \AND
  \bauthor{\bsnm{Reed},~\bfnm{B.}\binits{B.}}
(\byear{1995}).
\btitle{A critical point for random graphs with a given degree sequence}.
\bjournal{Random Structures Algorithms}
\bvolume{{\bf 6}}
\bpages{161--179}.
\bmrnumber{MR1370952 (97a:05191)}
\end{barticle}
\endbibitem

\bibitem{MolRee98}
\begin{barticle}[author]
\bauthor{\bsnm{Molloy},~\bfnm{M.}\binits{M.}} \AND
  \bauthor{\bsnm{Reed},~\bfnm{B.}\binits{B.}}
(\byear{1998}).
\btitle{The size of the giant component of a random graph with a given degree
  sequence}.
\bjournal{Combin. Probab. Comput.}
\bvolume{{\bf 7}}
\bpages{295--305}.
\bmrnumber{MR1664335 (2000c:05130)}
\end{barticle}
\endbibitem

\bibitem{Nach09}
\begin{barticle}[author]
\bauthor{\bsnm{Nachmias},~\bfnm{A.}\binits{A.}}
(\byear{2009}).
\btitle{Mean-field conditions for percolation on finite graphs}.
\bjournal{Geometric And Functional Analysis}
\bvolume{{\bf 19}}
\bpages{1171--1194}.
\bdoi{10.1007/s00039-009-0032-4}
\bmrnumber{2570320}
\end{barticle}
\endbibitem

\bibitem{NacPer07b}
\begin{barticle}[author]
\bauthor{\bsnm{Nachmias},~\bfnm{A.}\binits{A.}} \AND
  \bauthor{\bsnm{Peres},~\bfnm{Y.}\binits{Y.}}
(\byear{2008}).
\btitle{Critical random graphs: diameter and mixing time}.
\bjournal{Ann. Probab.}
\bvolume{{\bf 36}}
\bpages{1267--1286}.
\bdoi{10.1214/07-AOP358}
\bmrnumber{MR2435849 (2009h:05190)}
\end{barticle}
\endbibitem


\bibitem{Rior12}
\begin{barticle}[author]
\bauthor{\bsnm{Riordan},~\bfnm{O.}\binits{O.}}
(\byear{2012}).
\btitle{The phase transition in the configuration model}.
\bjournal{Combinatorics, Probability and Computing}
\bvolume{{\bf 21}}
\bpages{1--35}.
\bdoi{10.1017/S0963548311000666}
\bmrnumber{2900063}
\end{barticle}
\endbibitem

\bibitem{Sale13b}
\begin{barticle}[author]
\bauthor{\bsnm{Salez},~\bfnm{J.}\binits{J.}}
(\byear{2013}).
\btitle{Weighted enumeration of spanning subgraphs in locally tree-like
  graphs}.
\bjournal{Random Structures Algorithms}
\bvolume{{\bf 43}}
\bpages{377--397}.
\bdoi{10.1002/rsa.20436}
\bmrnumber{3094425}
\end{barticle}
\endbibitem

\end{thebibliography}
